\documentclass[11pt,a4paper]{article}
\usepackage[latin1]{inputenc}
\usepackage[T1]{fontenc}
\usepackage{amsmath}
\usepackage{amsfonts}
\usepackage{amssymb}
\usepackage{amsthm}
\usepackage{color}
\usepackage{graphicx}
\usepackage{tabularx}
\usepackage{geometry}                          % Redefine the size of the margin
\usepackage[square]{natbib}
\usepackage{ stmaryrd }                         % for the brackets [|   |] 
\usepackage{todonotes}                          % for comments on the margin
\reversemarginpar                               % comments in the left margin
\usepackage{hyperref}                           % Hyperlinks in the .pdf. Must remain the last called package.

\newtheorem{theorem}      {Theorem}[section]
\newtheorem*{theorem*}    {Theorem}
\newtheorem{proposition}  [theorem]{Proposition}

\newtheorem{lemma}        [theorem]{Lemma}

\newtheorem{remark}       [theorem]{Remark}

\newtheorem{corollary}    [theorem]{Corollary}

\newtheorem{hypotheses}   [theorem]{Assumptions}

\newcommand{\N}{\mathbb{N}}     %set of natural integers
\newcommand{\R}{\mathbb{R}}     %set of real numbers
\newcommand{\1}{\mathbf{1}}     %indicatrix
\renewcommand{\P}{\mathbb{P}}   %probability
\newcommand{\E}{\mathbb{E}}     %expectation
\newcommand{\F}{\mathcal{F}}    %set of functions
\renewcommand{\L}{\mathcal{L}}    %distribution of a rv 
 
\renewcommand{\a}{\alpha}       %birth rate
\renewcommand{\b}{\beta}        %death rate
\renewcommand{\l}{\lambda}      %parameter
\renewcommand{\b}{\beta}        %death rate
\newcommand{\p}{\partial}       %gradient
\renewcommand{\t}[1]{\widetilde{#1}}  %tilde
\renewcommand{\tilde}{\widetilde}
\renewcommand{\check}{\widehat}
\newcommand{\lip}{\text{Lip}}
\newcommand{\fw}[1]{\overset{\rightarrow}{#1}} %shift forward
\newcommand{\Fw}[1]{\overset{\longrightarrow}{#1}} %shift forward with long arrow
\newcommand{\bw}[1]{\overset{\leftarrow}{#1}} %shift backward
\newcommand{\Bw}[1]{\overset{\longleftarrow}{#1}} %shift backward with long arrow

\title{Intertwinings and Stein's magic factors for birth-death processes}
\author{Bertrand Cloez\thanks{UMR MISTEA, INRA Montpellier, France.} \and Claire Delplancke\thanks{Center of Mathematical Modeling, Chile. E-mail address: cdelplancke at cmm.uchile.cl (corresponding author).}}

\begin{document}

\maketitle

\begin{abstract} 

This article investigates second order intertwinings between semigroups of birth-death processes and discrete gradients on $\N$. It goes one step beyond a recent work of Chafa\"{i} and Joulin which establishes and applies to the analysis of birth-death semigroups a first order intertwining. Similarly to the first order relation, the second order intertwining involves birth-death and Feynman-Kac semigroups and weighted gradients on $\N$, and can be seen as a second derivative relation. As our main application, we provide new quantitative bounds on the Stein factors of discrete distributions. To illustrate the relevance of this approach, we also derive approximation results for the mixture of Poisson and geometric laws.


\paragraph{Mathematics Subject Classification (MSC2010):} 
60E15; 60J80, 47D08, 60E05, 60F05.
\end{abstract}

\tableofcontents
\section{Introduction}

A birth-death process is a continuous-time Markov process with values in $\N=\{0,1,\dots\}$ which evolves by jumps of two types: onto the integer just above (birth) or just below (death). We denote by BDP$(\a,\b)$ the birth-death process with positive birth rate $\a=(\a(x))_{x\in \N}$ and non-negative death rate $\b=(\b(x))_{x\in \N}$ satisfying to $\b(0)=0$. Its generator is defined for every function $f:\N \rightarrow \R$ as
\begin{align*}
Lf(x)&=\a(x)(f(x+1)-f(x))+\b(x) (f(x-1)-f(x)),&x\in \N.
\end{align*}

For a generator $L$, associated to a semigroup $(P_t)_{t\geq 0}$ and a Markov process $(X_t)_{t\geq 0}$ on $\N$, and a function $V$ on $\N$ (usually called a potential), the Schr\"odinger operator $L-V$ is defined for every function $f$ as 
\begin{align*}
(L -V) f(x)&=(L f)(x)-V(x)f(x)  ,&x\in \N,
\end{align*}
and is associated to the Feynman-Kac semigroup $(P_t^V)_{t \geq 0}$ defined for all bounded or non-negative functions $f$ on $\N$ as 
\begin{align*}
(P_t^V f)(x)&=\E\left[f(X_t^x)e^{-\int_0^t{V(X_s^x) ds}}\right],&x\in \N,\quad t\geq 0.
\end{align*}

The starting point of our work is the recent article \cite{chafai2013intertwining} which establishes a first order intertwining relation involving birth-death and Feynman-Kac semigroups, and discrete gradients on $\N$. For example, it reads as
\begin{equation}
\label{eq:interw1}
\p P_t  =\tilde{P}^{\tilde{V}}_t \p ,
\end{equation}
where $\p $ is the discrete gradient defined by $\p f(x) = f(x+1)-f(x)$, the notation $(\tilde{P}^{\tilde{V}}_t)_{t\geq 0}$ standing for an alternative Feynman-Kac semigroup. Actually, the precise result holds more generally for weighted gradients and allows to derive known as well as new results on the analysis of birth-death semigroups.\\

According to this observation, the aim of the present article is to extend this work by stating a second order intertwining relation. More precisely, let us define the backward gradient $\p^*$ by
\begin{align*}
\p^*f(x)&=f(x-1)-f(x), & x \in \N^*=\{1,2,\dots\};&& \p^*f(0)&=-f(0).
\end{align*}

Under some appropriate conditions on the potential $\t{V}$, we derive a formula of the type
\begin{equation}
\label{eq:interw2}
\p^* \p P_t  =\check{P}^{\check{V}}_t \p^* \p ,
\end{equation}
where $(\check{P}^{\check{V}}_t)_{t \geq 0}$ is a new Feynman-Kac semigroup. Similarly to the first order, this second order intertwining relation, which is our main result, is given in the more general case of weighted gradients.

Once our second order relation is established, it reveals to have many interesting consequences. In particular, we derive results on the estimation of the so-called Stein factors. Stein's factors, also known as Stein's magic factors, are upper bounds on derivatives of the solution to Stein's equation and a key point in Stein's method, introduced by Stein in \cite{stein1972bound}, which consists in evaluating from above distances between probability distributions. Among the important results appearing more or less recently in this very active field of research, let us cite some references within the framework of discrete probabilities distributions. Stein's factors related to the Poisson approximation in total variation and Wasserstein distances are studied in the seminal paper \cite{chen1975poisson}, in the reference book \cite{barbour1992poisson} and in the recent article \cite{barbour2015stein} for example. For the binomial negative approximation, one can cite \cite{brown1999negative} for the total variation distance and \cite{barbour2015stein} for the Wasserstein distance; for the geometric approximation in total variation distance, see \cite{pekoz1996stein} and \cite{pekoz2013total}. An important advance is made in \cite{brown2001stein}, where a universal approach to evaluate Stein's factors for the total variation distance is developed. The work \cite{eichelsbacher2008stein} provides Stein's factors for the total variation distance when the target distribution is a Gibbs distribution.  While our approach relies on the so-called generator method, which characterizes the reference distribution as the invariant measure of some Markov process, more general Stein operators have also been developed (\citep{LeyReinertSwan}).\\

%\cite{brown2001stein} developed a universal approach to bound the solution $g_f$ of the Stein's equation and its first derivative $\p g_f$ when $f$ is bounded. This is useful to quantify the total variation distance between probability measures; see for instance \cite{barbour1992poisson}, \cite{brown2001stein}, \cite{pekoz2013total}. To quantify the Wasserstein distance between probability measures, one needs upper bounds on the solution $g_f$ to Stein's equation for Lipschitz functions $f$. In this case, the techniques developed in the recent works of \cite{barbour2006stein,barbour2015stein} are specific to the underlying framework of interest (binomial negative and Poisson laws). Let us also cite \cite{pekoz2013total} who studied Stein's factors for the geometric distribution and for the classes of functions $\{\mathbf{1}_{m}, m \in \mathbb{N} \}$ and $\{\mathbf{1}_{[0,m]}, m \in \mathbb{N} \}$; these classes of functions correspond respecively to the local and Kolmogorov distances.
 
In the present article, we propose a universal technique to evaluate Stein's factors related to the approximation in total variation, Wasserstein and Kolmogorov distances. On the basis of some results derived in \cite{brown2001stein}, the main ingredients are the method of the generator and the intertwining relations presented above. To the authors' knowledge, the systematic use of this last ingredient, which comes from the functional analysis, seems to be new within the context of Stein's method. It allows to construct a unified framework for the derivation of Stein's factors, which applies to a wide range of discrete probability distributions-namely, distributions that are invariant with respect to some reversible birth-death process on $\N$ with good properties. A similar approach might be developed similarly for continuous distributions characterized as the invariant measure of some diffusion processes, for which a first order intertwining relation already exists (\cite{bonnefont2014intertwining,C12}); or for other discrete distributions, such as compound Poisson distributions which are invariant with respect to some downwards skip-free process.

A case-by-case examination of our general results in examples of interest reveals that our upper bounds sometimes improve on the ones already known, and sometimes are not as sharp. For example, we improve the first Stein factor related to the negative binomial approximation in total variation distance and we derive new Stein's factors for the geometric approximation in Wasserstein distance.

As an additional part of independent interest, we study the approximation of mixture of discrete distributions in the spirit of the Stein method. Combined with the Stein bounds, the obtained results have potential applications of which we give a flavour through the following example. Denote  NB$(r,p)$ the negative binomial distribution of parameters $(r,p)$. It is a mixed Poisson distribution, converging in law towards the Poisson distribution $\mathcal{P}_\l$ in the regime $p\to 1$, $r\to \infty$ and $r(1-p)/p\rightarrow \l$. The following bound in Wasserstein distance $W$ seems to be the first attempt to quantify this well-known convergence:
\begin{align*}
W\left(\text{NB}(r,p),\mathcal{P}_{\frac{r(1-p)}{p}}\right) 
&\leq \frac{8}{3\sqrt{2e}} \sqrt{\frac{r(1-p)}{p}}\frac{(1-p)}{p}. \\
\end{align*}

To conclude this introduction, let us announce the structure of the article. In Section \ref{sect:main}, we state with Theorem \ref{theorem_intertwining_order2_hard} our main result about the second order intertwining, after having recalled the first order intertwining; we follow with an application to the ergodicity of birth-death semigroups. In Section \ref{sect:stein},  we firstly present theoretical bounds on Stein's factors derived from the intertwinings, and secondly we investigate the approximation of mixture of distributions. In Section \ref{sect:exe}, our results are applied to a wide range of examples, including M/M/$\infty$ process and Poisson approximation, Galton-Watson process with immigration and negative binomial approximation, and M/M/$1$ process and geometric approximation. The three last sections are devoted to the various proofs of the results previously stated: Section \ref{sect:proof_main} deals with the preparation and proof of our main result Theorem \ref{theorem_intertwining_order2_hard}, Section \ref{sect:proof_stein} gathers the proofs of the bounds on Stein's factors and finally, a useful upper bound related to the pointwise probabilities of the M/M/$\infty$ process is proved in Section \ref{sect:proof_expl}.

 \ \\ \textbf{Acknowledgement:} 
 
The authors thank A. Joulin for the interest he took in this research through many fruitful discussions and for sharing his insights on the subject.

This work was partially supported by the CIMI (Centre International de Math\'{e}matiques et d'Informatique) Excellence program, by the ANR PIECE (ANR-12-JS01-0006-01) and STAB (ANR-12-BS01-0019) and by the Chaire Mod\'elisation Math\'ematique et Biodiversit\'e.

Part of this work has been done while the second author was affiliated to Institut de Math\'ematiques de Toulouse (UMR CNRS 5219) and Universit\'e of Toulouse, France.

\section{Main result}
\label{sect:main}

Before stating our main result Theorem \ref{theorem_intertwining_order2_hard}, let us introduce some notation. The set of positive integers $\left\{1,2,\dots \right\}$ is denoted $\N^*$. For all real-valued functions $f$ on $\N$ and sets $A \subset \N$, we define $\|f\|_{\infty,A}=\sup\left\{|f(x)|,x\in A\right\}$ and $\|f\|_{\infty}=\|f\|_{\infty,\N}$. For all sequences $u$ on $\N$, the shift-forward and shift-backward of $u$ are defined as:
\begin{align*}
\fw{u}(x)&=u(x+1), &x\in \N; && \bw{u}(x)&=u(x-1),& x\in \N^*;&& \bw{u}(0)&=0.
\end{align*}
The symbol $\mathcal{P}$ stands for the set of probability measures on $\N$ and we denote by $\L(W)$ the distribution of the random variable $W$. For all real-valued functions $f$ on $\N$ and $\mu \in \mathcal{P}$, we use indifferently the notation
$$\int{fd\mu}=\mu(f)= \sum_{x \in \mathbb{N}} f(x)\mu(x). $$
%For $x,y \in \mathbb{Z}$ such that $x>y$ and some sequence $(w_z)_{z\geq0}$, we use the conventions $\sum_{k=x}^y w_k=0$ and $\prod_{k=x}^y w_k=1$\Claire{N\'ecessaire?}.\\
Recall that the discrete forward and backward gradients are defined for all real-valued functions $f$ on $\N$ by 
\begin{align*}
\p f(x)&=f(x+1)-f(x),& x\in\N ; && \p^* f(x)&=f(x-1)-f(x),&x\in \N^*, &&\p^*f(0)&=-f(0),
\end{align*}
the convention chosen for $\p^*$ in $0$ being interpreted as a Dirichlet-type condition (implicitly we set $f(-1)=0$). Letting $u$ be a positive sequence, we define the weighted gradients $\p_u$ and $\p_u^*$ respectively by 
\begin{align*}
\p_u&=\frac{1}{u}\p, & \p_u^*&=\frac{1}{u}\p^*.
\end{align*}
With this notation, the generator of the BDP$(\a,\b)$ reads for every function $f:\N \rightarrow \R$ as
\begin{align*}
Lf&=\a\,\p f+\b \,\p^*f.
\end{align*}
Let us assume that the birth rate $\a$ is positive on $\N$ and that the death rate $\b$ is positive on $\N^*$ with moreover $\b(0)=0$. Hence the process is irreducible; to ensure that the process is ergodic and non-explosive we further assume respectively that (\cite{dobrushin1952conditions}, \cite[Corollary 3.18]{chen2004markov})
\begin{align*}
\sum_{x=1}^{+\infty}{\frac{\a(0)\a(1)\dots\a(x-1)}{\b(1)\b(2)\dots \b(x)}}&< \infty, & \sum_{x=1}^\infty{\left(\frac{1}{\a(x)}+\frac{\b(x)}{\a(x)\a(x-1)}+\dots+\frac{\b(x)\dots\b(1)}{\a(x)\dots\a(0)}\right)}&=\infty.
\end{align*}
The measure $\pi$ defined on $\N$ as 
\begin{align}
\label{eq:invariant}
\pi(0)&= \left(1+\sum_{x\geq 1} \prod_{y=1}^x{ \frac{\a(y-1)}{\b(y)}} \right)^{-1}, & \pi(x)&=\pi(0)\prod_{y=1}^x{\frac{\a(y-1)}{\b(y)}}, & x\in \N ,
\end{align} 
is then the invariant, and symmetric, probability measure for the associated semigroup.

Recall that if $(P_t)_{t\geq 0}$ is a Markov semigroup on $\N$ associated to the process $(X_t)_{t\geq 0}$ and if the potential $V:\N\rightarrow\R$ is bounded from below, the Feynman-Kac semigroup $(P_t^V)_{t \geq 0}$ is defined for all bounded or non-negative functions $f$ on $\N$ as 
\begin{align}
(P_t^V f)(x)&=\E\left[f(X_t^x)e^{-\int_0^t{V(X_s^x) ds}}\right],&x\in \N,\quad t\geq 0.
\label{equation_feynman_kac_semigroup}
\end{align} 
When $V$ is positive, the formula \eqref{equation_feynman_kac_semigroup} admits an interpretation involving a killed, or extended, Markov process. Add a new state $a$ to $\N$ and extend functions $f$ on $\N$ to $\N\cup\left\{a\right\}$ by $f(a)=0$. Then, we have:
\begin{align*}
P_t^Vf(x)&=\E\left[f(Y_t^x)\1_{\left\{Y_t^x \neq a \right\}}\right],
\end{align*}
where the process $(Y_t^x)_{t\geq 0}$ is absorbed in $a$ with rate $V(Y_t^x)$. The generator of the process $(Y_t^x)_{t\geq0}$ acts on real-valued functions on $\N\cup\left\{a\right\}$ by the formula 
\begin{equation}
\label{eq:gen_killed}
(Kf)(x)=(Lf|_\N)(x)+V(x)(f(a)-f(x)).
\end{equation}
This interpretation can be extended to the case where $V$ is bounded from below by adding and subtracting a constant to $V$ inside the exponential. 

The Kolmogorov equations associated to the Schr\"odinger operator $L-V$ and the Feynman-Kac semigroup defined in the introduction read for all functions $f$ in the domain of $L$ as
\begin{align}
\p_t P_t^Vf&=(L-V)P_t^V f=P_t^V(L-V)f,& t\geq 0.
\label{equation_Kolomogorov_Feynamn-Kac}
\end{align}
Here $\p_t$ denotes the derivative in time. In the following, when using a Feynman-Kac semigroup, we will always assume that the equation (\ref{equation_Kolomogorov_Feynamn-Kac}) stands for all bounded real-valued functions on $\N$. It is the case for example when $L$ is the generator of a birth-death process with rates $(\a,\b)$, and $\a,\b,V$ are $P_t$-integrable for all $t\geq 0$. \\

In order to state the first intertwining relation, we associate to any positive sequence $u$ a modified birth-death process on $\N$ with semigroup $(P_{u,t})_{t\geq 0}$, generator $L_u$, and potential $V_u$. For all functions $f:\N \rightarrow \R$ set
\begin{align*}
L_u f&= \a_u\, \p f+\b_u\, \p^*f, &V_u&=\a-\a_u+\fw{\b}-\b_u,&\\
\a_u(x)&=\frac{u(x+1)}{u(x)}\a(x+1), & \b_u(x)&=\frac{u(x-1)}{u(x)}\b(x)\1_{x\in \N^*},&x\in \N.
\end{align*}

Under the compacted form $V_u=\p_u\left(\bw{u}\b-u\a \right)$ one can see the parallel with the analogous formulas in the diffusion setting (\cite{bakry1985diffusions, bonnefont2014intertwining}).

%The semigroup $(P_{u,t})_{t\geq 0}$ is related to the gradient $\p_u$., another semigroup is introduced in Section \ref{sect:backward} for the gradient $\p_u^*$. 
We recall now the first order intertwining relation, due to \cite{chafai2013intertwining}.

\begin{theorem}[First order intertwining relation]
\label{theorem_intertwining_order1}
If $V_u$ is bounded from below, it holds for every real-valued function on $\N$ such that $\|\p_uf\|_\infty <+\infty$ that:
\begin{align}
\p_u P_t f&= P_{u,t}^{V_u}\, \p_u f,&t\geq 0.
\label{formula_intertwining_sg_order1}
\end{align}
\end{theorem}

Although we will not prove this result in full generality, a new proof is proposed in Section \ref{sect:proofcoupling} when the weight is $u=1$, the birth rates $\a$ are non-increasing and the death rates $\b$ are non-decreasing. This proof is based on a coupling argument and gives a probabilistic interpretation of the semigroup (and its jump rates) in the right-hand side of \eqref{formula_intertwining_sg_order1}.\\

We now turn to the main theorem of this article. Let $u$ and $v$ be positive sequences and assume that the potential $V_u$ defined above is non-increasing on $\N$. We define a modified process on $\N$ with semigroup $(P_{u,*v,t})_{t\geq 0}$ and generator $L_{u,*v}$ as follows: for all real-valued functions $f$ on $\N$, set
\begin{align*}
(L_{u,*v}f)(x)&=\a_{u,*v}(x)\p f(x)+\b_{u,*v}(x)\p^* f(x)&\\
&+(\p_v^* V_u)(x)\left(\sum_{j=0}^{x-2}{v(j)}\right)\sum_{k=0}^{x-2}{\frac{v(k)}{\left(\sum_{j=0}^{x-2}{v(j)}\right)}(f(k)-f(x))},& x &\geq 2,\\
(L_{u,*v}f)(x)&=\a_{u,*v}(x)\p f(x)+\b_{u,*v}(x)\p^* f(x),& x&=0,1,\\
\a_{u,*v}(x)&=\frac{v(x+1)}{v(x)}\frac{u(x+1)}{u(x)}\a(x+1),&x&\in \N,\\
\b_{u,*v}(x)&=\frac{v(x-1)}{v(x)}\frac{u(x-2)}{u(x-1)}\b(x-1)+v(x-1)\,\p_v^* V_u(x),& x&\geq 2,\\
\b_{u,*v}(1)&=v(0)\,\p_v^* V_u(1), \quad\b_{u,*v}(0)=0.&&
\end{align*}
In contrast with the previous semigroups, this modified process is not a birth-death process in general. Indeed, if the process starts at a point $x \geq 2$, it can jump on the set $\left\{0,\dots ,x-2 \right\}$ with rate $(\p_v^* V_u)(x)\left(\sum_{j=0}^{x-2}{v(j)}\right)$. Remark that both this quantity and the death rate in $1$, $\b_{u,*v}(1)=(\p^*V_u)(1)$, are non-negative thanks to the hypothesis $V_u$ non-increasing on $\N$. We also define the potential $V_{u,*v}$ as 
\begin{align*}
V_{u,*v}(x)&=\left(1+\frac{u(x)}{u(x-1)}\right)\a(x)-\left(1+\frac{v(x+1)}{v(x)}\right)\frac{u(x+1)}{u(x)}\a(x+1)&\\
&+\b(x+1)-\frac{v(x-1)}{v(x)}\frac{u(x-2)}{u(x-1)}\b(x-1)-\left(\sum_{j=0}^{x-1}{v(j)}\right)\p_v^*V_u(x),& x \geq 1,\\
V_{u,*v}(0)&=\a(0)-\left(1+\frac{v(1)}{v(0)}\right)\frac{u(1)}{u(0)}\a(1)+\b(1).&
\end{align*}
 
We are ready to state our main result.

\begin{theorem}[Second order intertwining relation]
\label{theorem_intertwining_order2_hard}
Assume that $V_u$ is non-increasing, bounded from below, that $\inf_{x\in \N} v(x)>0$ and that $V_{u,*v}$ is bounded from below. Then for every real-valued function on $\N$ such that $\|\p_uf\|_\infty <+\infty$, we have
\begin{align*}
\p_{v}^*\p_u( P_t f)&= P_{u,*v,t}^{V_{u,*v}} \,(\p_{v}^*\p_u f),&t\geq 0.
\end{align*}
\end{theorem}

Since some preparation is needed, the proof of Theorem \ref{theorem_intertwining_order2_hard} is postponed to Section \ref{sect:proof_main}.

\begin{remark}[Propagation of convexity ?]
\label{rq:convexity}
Under the assumptions of Theorem \ref{theorem_intertwining_order2_hard}, if $\p_v^*\p_u f$ is non-negative, so is $\p_v^*\p_u P_t f$ for all $t\geq 0$. A similar property for the first order intertwining admits an interpretation in terms of propagation of monotonicity (\cite[Remark 2.4]{chafai2013intertwining}): the intertwining relation \eqref{formula_intertwining_sg_order1} implies that if a function $f:\N\rightarrow\R$ is non-decreasing, then so is $P_t f$ for every $t\geq 0$. However, it is not clear whether there is an analogous nice interpretation for the second order intertwining because, in contrast to the continous space case, the condition $\p_v^*\p_u f\geq0$ is not equivalent to the convexity of $f$ (even for $u=v=1$).
\end{remark}

Let us comment further on Theorem \ref{theorem_intertwining_order2_hard}. The interpretation of a Feynman-Kac semigroup as an extended Markov semigroup sheds light on various aspects of Theorem \ref{theorem_intertwining_order2_hard}. As the first-order potential $V_u$ is bounded from below, recall that the Feynman-Kac semigroup $(P_{u,t}^{V_u})_{t\geq 0}$ appearing in the right-hand side of equation \eqref{formula_intertwining_sg_order1} can be represented as a Markov semigroup $(S_t)_{t\geq 0}$ related to the process $(Y_t)_{t\geq 0}$ on $\N \cup \left\{-1\right\}$ by adding a point $a=-1$. The Markov process $(Y_t)_{t\geq 0}$ is then non-irreducible and absorbed in $-1$. To differentiate again in the equation \eqref{formula_intertwining_sg_order1} amounts to differentiate the Markov semigroup $(S_t)_{t\geq 0}$.

Firstly, this explains intuitively the use of the backward weighted gradient $\p_u^*$ instead of the regular weighted gradient $\p_u$.  Indeed, to deal with the absorption of the Markov process in $-1$, additional information at the boundary is needed. The use of $\p^*$ gives the missing information, since the knowledge of $\p^*g$ is equivalent to the knowledge of $\p g$ in addition with the knowledge of $g(0)=-\p^* g(0)$. 

Secondly, this allows to understand the hypotheses required for Theorem \ref{theorem_intertwining_order2_hard} to apply. The main assumption of this theorem is $V_u$ to be non-increasing. As noticed before, this assumption is necessary in order to have well-defined objects. The following remark provides another justification.

\begin{remark}[Around the monotonicity assumption]
\label{rq:ajoutdepoint}
On the one hand, the second intertwining relation is equivalent to a first intertwining relation for the extended Markov semigroup $(S_t)_{t\geq 0}$. On the other hand, if a first intertwining relation holds for $(S_t)_{t\geq 0}$, then $(S_t)_{t\geq 0}$ propagates the monotonicity. Set $f= \mathbf{1}_{\N} = 1- \mathbf{1}_{\{-1\}}$. Then for all $x,y\in \N$, $S_0 f(x)=S_0 f(y)=1$ and by formula (\ref{eq:gen_killed}),
\begin{eqnarray*}
\p_t(S_t f)(x)|_{t=0}&=&(L_u f|_\N )(x) +V_u(x) (f(-1)-f(x))=-V_u(x),\\
\p_t\big(S_t f(x)-S_t f(y)\big)|_{t=0}&=&V_u(y)-V_u(x).
%\p_t(S_t f)(x)|_{t=0}-\p_t (S_t f)(y)|_{t=0}&=&V_u(y)-V_u(x).
\end{eqnarray*} 
 
The function $f$ is non-decreasing on $\N \cup \left\{-1\right\}$ and a necessary condition for $S_t f$ to be non-decreasing for all $t\geq 0$ is, in the light of the preceding equation, that $V_u(y)-V_u(x) \leq 0$ whenever $x \leq y$, i.e. $V_u$ is non-increasing on $\N$.\\
\end{remark}

If $V_u$ is constant, then Theorem \ref{theorem_intertwining_order2_hard} admits a variant involving the gradient $\p_v \p_u$ instead of $\p_v^* \p_u$, which is stated in Theorem \ref{theorem_intertwining_order2_easy} below for the sake of completeness. In the applications, when $V_u$ is constant, we choose to invoke Theorem \ref{theorem_intertwining_order2_easy} in lieu of Theorem \ref{theorem_intertwining_order2_hard}, because the underlying arguments are much simpler.
Indeed, in this case the equation \eqref{formula_intertwining_sg_order1} reduces to 
\begin{align*}
\p_u P_t f&= e^{-V_u t} P_{u,t} \p_u f,&t\geq 0,
\end{align*}
and it is no longer required to extend artificially the Markov process, nor to add information at the boudary, in order to differentiate a second time. As a matter of fact, one can notice that if $V_u$ is constant, then the BDP associated to the semigroup $(P_{u,*v,t})_{t\geq 0}$ of Theorem \ref{theorem_intertwining_order2_hard} do not visit the state $0$ unless it starts there. 

In order to state the theorem, a new birth-death semigroup $(P_{u,v,t})_{t\geq 0}$ with generator $L_{u,v}$ and a potential $V_{u,v}$ is introduced. Set for all real-valued functions on $\N$:
\begin{align*}
L_{u,v}f(x)&=\a_{u,v}\p f(x)+\b_{u,v}\p^* f(x),&x\in \N,\\
\a_{u,v}(x)&=\frac{v(x+1)}{v(x)}\frac{u(x+2)}{u(x+1)}\a(x+2),\quad \b_{u,v}(x)=\frac{v(x-1)}{v(x)}\frac{u(x-1)}{u(x)}\b(x),&x\in \N,\\
V_{u,v}(x)&=\a(x)-\frac{v(x+1)}{v(x)}\frac{u_{x+2}}{u(x+1)}\a(x+2)+\left(\frac{u(x)}{u(x+1)}+1\right)\b(x+1)&\\
&-\left(1+\frac{v(x-1)}{v(x)}\right)\frac{u(x-1)}{u(x)}\b(x),&x\in \N.
\end{align*}
In contrast to the Markov semigroup $(P_{u,*v,t})_{t\geq 0}$, the semigroup $(P_{u,v,t})_{t\geq 0}$ is always  a birth-death semigroup.

\begin{theorem}[Alternative version of the second intertwining relation]
\label{theorem_intertwining_order2_easy} 
Assume that $V_u$ is constant on $\N$ and that $V_{u,v}$ is bounded from below. For all real-valued functions on $\N$ such that $\|\p_uf\|_\infty <+\infty$ and $\|\p_v\p_uf\|_\infty <+\infty$, we have:
\begin{align*}
\p_{v}\p_u( P_t f)&= P_{u,v,t}^{V_{u,v}}\, (\p_{v}\p_u f),&t\geq 0.\\
\end{align*}
\end{theorem}

\begin{remark}[Link between the two versions of the second intertwining]
\label{rq:link_easy_hard}
Surprisingly, it is only possible to deduce directly Theorem \ref{theorem_intertwining_order2_easy} from Theorem \ref{theorem_intertwining_order2_hard} in the case where the sequence $v$ is constant. When $v=1$ for instance, one can write that $\p^* \p_u f(\cdot+1)=-\p\p_u f$, yielding under the appropriate assumptions on $f:\N\rightarrow\R$ that:
\begin{align*}
P_{u,1,t}f(x)&=P_{u,*1,t}\overleftarrow{f}(x+1), &x\in \N,\quad t\geq 0.
\end{align*}
%\label{eq:sg_easy_hard}
At the level of the processes, this equation can be reformulated into the equality in law:
\begin{align*}
X_{1,u,t}^x&=X_{1,*u,t}^{x+1}-1,&x\in \N,\quad t\geq 0,
\end{align*} 
where $(X_{1,*u,t}^x)_{t\geq 0}$ and $(X_{1,u,t}^x)_{t\geq 0}$ are the Markov processes corresponding respectively to the semigroups $(P_{u,*1,t})_{t\geq 0}$ and $(P_{u,1,t})_{t\geq 0}$.
If $v$ is not constant, no similar relation holds in general.
\end{remark}

\begin{remark}[Other versions]
It is possible to derive similar theorems for other gradients. For example, if the gradient $\p^\star$ is defined as $\p^\star f=\p^* f$ on $\N^*$ and with the Neumann-like boundary condition in $0$, $\p^\star f(0)=0$, then the analogous theorem to Theorem \ref{theorem_intertwining_order2_hard} holds for $\p_v\p_u^\star$. It is also possible to derive intertwining relations in the case where the semigroup lives on $\llbracket 0,n\rrbracket$, although the underlying structures are rather different: for instance, the condition $V_u$ non-increasing is no longer necessary.\\
\end{remark}

Let us turn to our first application of Theorem \ref{theorem_intertwining_order2_hard} and its variant Theorem \ref{theorem_intertwining_order2_easy}. The first order intertwining relation recalled in Theorem \ref{theorem_intertwining_order1} yields a contraction property in Wasserstein distance. Precisely, under the assumptions of Theorem \ref{theorem_intertwining_order1}, by \cite[Corollary 3.1]{chafai2013intertwining}, we have for all $\mu,\nu\in \mathcal{P}$,
\begin{equation}
\label{eq:contraction_wasserstein}
W_{d_u}(\mu P_t,\nu P_t)\leq e^{-\sigma(u)t}W_{d_u}(\mu,\nu),
\end{equation}
where the distance $d_u$ on $\N$ and the related Wasserstein distance $W_{d_u}$ on $\mathcal{P}$ are defined in the forthcoming section, Section \ref{sect:dist}. Similarly, Theorems \ref{theorem_intertwining_order2_hard} and \ref{theorem_intertwining_order2_easy} lead to a contraction property for the distances $\zeta_{u,*v}$ and $\zeta_{u,v}$, defined respectively for two sequence of positive weights $u$ and $v$ by
%\begin{align*}
%\zeta_{u,*v}&=\sup_{f \in \mathcal{F}_{u,*v}}|\mu(f)-\nu(f)|,& \zeta_{u,v}&=\sup_{f \in \mathcal{F}_{u,v}}|\mu(f)-\nu(f)|, \\
%\mathcal{F}_{u,*v}&=\left\{f:\N\rightarrow\R,\|\p_v^* \p_uf\|_\infty\leq 1\right\}, & \mathcal{F}_{u,v}&=\left\{f:\N\rightarrow\R,\|\p_v \p_uf\|_\infty\leq 1\right\}.
%\end{align*}
\begin{align*}
\zeta_{u,*v}&=\sup_{f \in \mathcal{F}_{u,*v}}|\mu(f)-\nu(f)|,& \mathcal{F}_{u,*v}&=\left\{f:\N\rightarrow\R,\,\|\p_u f\|_\infty <\infty,\,\|\p_v^* \p_uf\|_\infty\leq 1\right\},\\
\zeta_{u,v}&=\sup_{f \in \mathcal{F}_{u,v}}|\mu(f)-\nu(f)|,& \mathcal{F}_{u,v}&=\left\{f:\N\rightarrow\R,\,\|\p_u f\|_\infty <\infty,\,\|\p_v \p_uf\|_\infty\leq 1\right\}.
\end{align*}
We call $\zeta_{u,*v}$ and $\zeta_{u,v}$ second order Zolotarev-type distances since they are simple metric distances in the sense of Zolotarev (\cite{zolotarev1976metric}) and can be seen as the discrete counterparts of the distance $\zeta_{2}$ defined on the set of real probability distributions (the distance $\zeta_{2}$, introduced in \cite{zolotarev1976metric} and further studied in \cite{rio1998distances}, is associated to the set of continuously differentiable functions on $\R$ whose derivative is Lipschitz). The contraction property reads as follows:

\begin{theorem}[Contraction of the BDP in second order distances]
\label{theorem_contraction_semigroup_zolotarev_distance}
\begin{itemize}\
\item Under the same hypotheses as in Theorem \ref{theorem_intertwining_order2_hard}, we set $\sigma(u,*v)=\inf V_{u,*v}$. Then, for all $\mu,\nu \in \mathcal{P}$, we have:
\begin{align}
\zeta_{u,*v}(\mu P_t, \nu P_t)&\leq e^{-\sigma(u,*v)t}\zeta_{u,*v}(\mu, \nu).
\label{eq:contraction_zolo_1}
\end{align}

\item Under the assumptions of Theorem \ref{theorem_intertwining_order2_easy}, define $\sigma(u,v)=\inf V_{u,v}$. Letting $\mu,\nu \in \mathcal{P}$, it stands that:
\begin{align}
\zeta_{u,v}(\mu P_t, \nu P_t)&\leq e^{-\sigma(u,v)t}\zeta_{u,v}(\mu, \nu).
\label{eq:contraction_zolo_2}
\end{align}
\end{itemize}  
\end{theorem}

\begin{proof}
The proof is done in the first case, the second one being similar. For all real-valued functions $f$ on $\N$ such that $\|\p_u f\|_\infty < \infty$ and $\|\p_v^* \p_u f\|_\infty \leq 1$, Theorem \ref{theorem_intertwining_order2_hard} implies that
\begin{align*}
\|\p_v^* \p_u P_t f\|_\infty &\leq \, e^{-\sigma(u,*v)t} \|\p_v^* \p_uf\|_\infty  \,\leq \, e^{-\sigma(u,*v)t}, & t\geq 0.
\end{align*} 
 Hence,
\begin{eqnarray*}
\zeta_{u,*v}(\mu P_t, \nu P_t)&=&\sup\left\{\left|\int{P_tf \,d\mu}-\int{P_tf \,d\nu}\right|,\,{\|\p_v^* \p_u f\|_\infty\leq 1}\right\}\\
&\leq & \sup\left\{\left|\int{g \,d\mu}-\int{g \,d\nu}\right|,\,{\| g\|_\infty\leq e^{-\sigma(u,*v)t}}\right\}\\
&=&e^{-\sigma(u,*v)t}\zeta_{u,*v}(\mu, \nu).
\end{eqnarray*}
\end{proof}

If the quantity $\sigma(u,*v)$ (resp. $\sigma(u,v)$) is positive, the first bound (resp. the second) is a contraction. In particular, if we take $\nu=\pi$ the invariant measure of the BDP, then Theorem \ref{theorem_contraction_semigroup_zolotarev_distance} gives the rate of convergence of the BDP towards its invariant measure in a second order distance.

\begin{remark}[Generalization and optimality]\
\label{rq:zolo-gen}
\begin{itemize}
\item The proof of Theorem \ref{theorem_contraction_semigroup_zolotarev_distance} can be generalized to the Zolotarev-type distance associated to the set of functions $f:\N\rightarrow \R$ such that $\|Df\|_\infty\leq 1$ as soon as we have an inequality of the type $\|D P_t f \|_\infty \leq e^{- \sigma t} \| Df \|_\infty$ for every $t\geq0$, some $\sigma>0$ and some finite difference operator $D$. In Section \ref{sect:exe} below, we detail such convergences in higher order Zolotarev-type distances.

\item By arguments similar to those developed in \cite[corollary 3.1]{chafai2013intertwining}, one can prove that the constants $\sigma(u,*v)$ and $\sigma(u,v)$ in the equations \eqref{eq:contraction_zolo_1} and \eqref{eq:contraction_zolo_2} are optimal. Indeed, the argument of \cite{chafai2013intertwining} relies on the propagation of the monotonicity and we have the analogous property at the second order (cf Remark \ref{rq:convexity}).

\item Using \cite[Theorem 9.25]{chen2004markov}, we see that, choosing a good sequence $u$, it is possible to obtain the contraction in the Wasserstein distance \eqref{eq:contraction_wasserstein} at a rate corresponding to the spectral gap (even if there is no corresponding eigenvector). For the second order, we do not know if it is always possible to find sequences $u,v$ such that $\sigma(*u,v)$ or $\sigma(u,v)$ is equal to the second smallest positive eigenvalue of $-L$. 

\end{itemize}
\end{remark}

In the following section we focus our attention on our main application of intertwining relations, Stein's factors.
\section{Application to Stein's magic factors}

\label{sect:stein}
\subsection{Distances between probability distributions}
\label{sect:dist}
First of all, we introduce the distances between probability measures used to measure approximations in the sequel. They are of the form
$$\zeta_{\mathcal{F}}(\mu,\nu)=\sup \left\{\left|\mu(f)-\nu(f) \right|,\,f \in \mathcal{F}\right\},$$
where $\mathcal{F}$ is a subset of the set of real-valued functions on $\N$. The distances $\zeta_{u,*v}$, $\zeta_{u,v}$ presented at the end of the preceding section were examples of such distances; we now recall the definition of three classical distances on $\mathcal{P}$.\\

\textbf{Total variation distance.}
The total variation distance $d_{\mathrm{TV}}$ is the distance associated to the set $\mathcal{F}_{\mathrm{TV}}$ of real-valued functions on $\N$ such that $0 \leq f \leq 1$. In contrast to the continous space case, the topology induced by the total variation distance on $\N$ is exactly the convergence in law. Some authors prefer to define the total variation distance as the distance associated to the set $\mathcal{F}=\left\{f:\N\rightarrow\R, \|f\|_\infty \leq 1\right\}$. The two definitions vary by a factor $\frac{1}{2}$:
\begin{align*}
d_{\mathrm{TV}}(\mu,\nu)&=\sup_{0\leq f \leq 1}\left|\mu(f)-\nu(f) \right|=\frac{1}{2}\sup_{\|f\|_\infty \leq 1}\left|\mu(f)-\nu(f) \right| = \frac{1}{2} \sum_{x \in \mathbb{N}} |\mu(x) - \nu(x)|.
\end{align*}

%\Clairein{D\'efinir la variation totale \`a poids ?}
%More generally, for all positive sequence $u$ we define the weighted total variation distance $d_{\mathrm{TV},u}$ as the metric distance associated to the set of real-valued functions on $\N$ such that $0 \leq \frac{f}{u} \leq 1$.\\%\newline
%\Bertrandin{$\frac{1}{2} \sum_{x \in \mathbb{N}} u(x)|\mu(x) - \nu(x)|$}

\textbf{Wasserstein distance.}
For a distance $d$  on $\N$ let us call $\text{Lip}(d)$ the set of real-valued functions on $\N$ such that 
\begin{align*}
|f(x)-f(y)|&\leq d(x,y),&x,y \in \N.
\end{align*}
The Wasserstein distance between two probability measures $\mu$ and $\nu$ of $\mathcal{P}$ is defined as 
$$W_d(\mu,\nu)=\inf{\int{d(x,y)d\Pi(x,y)}},$$ 
where the infimum is taken over all probability measures $\Pi$ on $\N^2$ whose first marginal is $\mu$ and second marginal is $\nu$. By Kantorovich-Rubinstein theorem (see e.g. \cite{szulga1982minimal}), 
$$W_d(\mu,\nu)=\zeta_{\text{Lip}(d)}(\mu,\nu).$$ 
For a positive sequence $u$, define the distance $d_u$ on $\N$ as 
\begin{align*}
d_u(x,y)&=\sum_{k=x}^{y-1}{u(k)},\quad x< y;& d_u(x,y)&=d_u(y,x), \quad x>y;& d_u(x,y)&=0,\quad  x=y.
\end{align*}
%$$d_u(x,y)=\sum_{k=x}^{y-1}{u(k)}\text{ if }x< y,\quad d_u(x,y)=d_u(y,x)\text{ if } x>y,\quad d_u(x,y)=0\text{ if } x=y.$$
Let us observe that $\text{Lip}(d_u)=\left\{f:\N\rightarrow\R,\|\p_u f\|_\infty\leq 1\right\}$. Hence
$$W_{d_u}(\mu,\nu)=\sup_{f\in \lip(d_u)}\left|\mu(f)-\nu(f) \right|=\sup_{\|\p_u f\|_\infty\leq 1}\left|\mu(f)-\nu(f) \right|.$$
The distance associated to the constant sequence equal to $1$ is the usual distance $d_1(x,y)=|x-y|$. We denote by $W=W_{d_1}$ the associated Wasserstein distance. \\
%The following alternative characterization holds:
%$$
%W(\mu,\nu) = \sum_{x\in \mathbb{N}} \left|\mu([0,x]) - \nu([0,x])\right|.
%$$ 

\textbf{Kolmogorov distance.}  The Kolmogorov distance is defined as the metric distance associated to the set $\mathcal{F}_{K}$ of indicator functions of intervals $[0,x]$:
\begin{align*}
d_K(\mu,\nu)&=\sup_{x\in \N}{\left|\mu([0,x])-\nu([0, x])\right|}.
\end{align*}

\textbf{Comparison between distances.} For all $\mu,\nu \in \mathcal{P}$,
\begin{equation*}
d_K(\mu,\nu) \,\leq \,d_{\mathrm{TV}}(\mu,\nu)\,\leq \,\frac{1}{\inf_\N u} W_{d_u}(\mu,\nu).
\end{equation*}
%\label{eq:majo_stein_wasserstein}

Indeed, both inequalities are consequences of the inclusions $$\mathcal{F}_{K} \subset \mathcal{F}_{\mathrm{TV}} \subset \frac{1}{\inf_\N u}\text{Lip}(d_u).$$
The second inclusion follows from the implication
\begin{equation*}
0\leq f \leq 1 \Rightarrow \|\p f\|_\infty\leq \frac{1}{\inf_\N u}.
\end{equation*}

The total variation distance is invariant by translation, whereas intuitively the Wasserstein distance gives more weight to the discrepancy between $\mu(x),\nu(x)$ if it occurs for a large integer $x$. The Kolmogorov distance may be used as an alternative to the total variation distance when the latter is too strong to measure the involved quantities.

\subsection{Basic facts on Stein's method}
Given a probability measure $\mu$ and a target probability measure $\pi$ of $\mathcal{P}$, the Stein-Chen method provides a way to estimate the distances of the type $\zeta_{\mathcal{F}}(\mu,\pi)$. More precisely, consider a Stein's operator $S$:
\begin{align*}
Sf(x)&=\a(x) f(x+1)-\b(x) f(x),&x\in \N&;& \b_0=0,
\end{align*}
characterizing the probability measure $\pi$ (meaning that $\int{Sf d\mu}=0$ for every function $f:\N\rightarrow\R$ in a sufficiently rich class of functions if and only if $\mu=\pi$) and the associated Stein equation
\begin{equation}
\label{equation_stein}
Sg_f=f-\int{fd\pi}.
\end{equation}
We call $g_f$ a solution to Stein's equation. The interest of such solutions comes from the following error bound:
\begin{align}
\zeta_{\mathcal{F}}(\mu,\pi)
&=\sup_{f \in \mathcal{F}} \left|\mu(f)-\pi(f) \right|= \sup_{f \in \mathcal{F}} \left| \int S g_f d\mu\right|.
\label{eq:error_bound}
\end{align}
As a consequence, if it can be shown that
$$
\left| \int S g_f d\mu \right| \leq \varepsilon_0 \| g_f\|_\infty + \varepsilon_1 \|\p g_f\|_\infty,
$$
then it follows that
$$
\zeta_{\mathcal{F}}(\mu,\pi) \leq \varepsilon_0 \sup_{f\in \mathcal{F}} \| g_f\|_\infty + \varepsilon_1 \sup_{f\in \mathcal{F}} \|\p g_f\|_\infty .
$$
This strategy of proof is widely used, for example in the references about Stein's method provided in the introduction.\\

A key point of this approach consists then in evaluating the so-called first and second Stein factors, also known as \textit{magic factors}:
\begin{align*}
\sup_{f\in \mathcal{F}}\| g_f\|_\infty,&& \sup_{f\in \mathcal{F}}\|\p g_f\|_\infty.
\end{align*}

%or more generally
%\begin{align*}
%\sup_{f\in \mathcal{F}}\Vert g_f /u \Vert_\infty,&& \sup_{f\in \mathcal{F}}\|\p_u g_f\|_\infty.
%\end{align*}

Observe that the equation \eqref{equation_stein} does not determine the value of $g_f(0)$. When evaluating the first Stein factor $\sup_{f\in \F}{\|g_f\|_\infty}$, we pick for every $f\in \F$ the solution $g_f$ of \eqref{equation_stein} such that $g_f(0)=0$. Hence, it is sufficient to consider the quantity
\begin{align*}
\sup_{f\in \F}{\|g_f\|_{\infty,\N^*}}&=\sup_{f\in \F}{\|\Fw{g_f}\|_{\infty}}.
\end{align*}
Similarly, for the second Stein factor, picking solutions $g_f$ to \eqref{equation_stein} satisfying to $g_f(0)=g_f(1)$, i.e. $\p g_f(0)=0$, allows to consider only the quantity
\begin{align*}
\sup_{f\in \F}{\|\p g_f\|_{\infty,\N^*}}&=\sup_{f\in \F}{\|\p \Fw{g_f}\|_{\infty}}.
\end{align*}

To evaluate the above quantities, we use a method known as method of the generator and the semigroup representation deriving from it. Set $L$ the generator and $(P_t)_{t\geq 0}$ the semigroup associated to the BDP$(\a,\b)$ and assume that $(P_t)_{t\geq 0}$ is invariant with respect to the target probability distribution $\pi$. The operators $S$ and $L$ are linked by the relation
$$
L h = S(-\partial^* h).
$$
The Poisson equation reads as
\begin{equation*}
Lh_f=f-\mu(f),
\end{equation*}
the centered solution $h_f$ being given by the expression
\begin{align*}
h_f &= -\int_0^\infty (P_tf - \mu(f)) dt.
\end{align*}
% \label{equation_poisson1}

Then, we obtain a solution $g_f$ to Stein's equation \eqref{equation_stein} under the so-called semigroup representation: 
\begin{align}
g_f &=-\p^* h_f=\int_0^\infty{\p^* P_t f dt}.
\label{eq:stein_sg_representation}
\end{align}

\subsection{Bounds on Stein's magic factors}
\label{sect:stein-gen}

In this section, theoretical bounds on the first and second order Stein factors are proposed for the approximation in total variation, Wasserstein and Kolmogorov distances. Proofs are postponed to Section \ref{sect:proof_stein} in order to clarify the presentation. Before turning to the results, a few general comments are made.

\begin{enumerate}
\item Our method evaluates Stein factors by quantities of the form
\begin{align*}
\int_{0}^\infty{e^{-\kappa t}\sup_{i \in \N}\P(\t{X}_t^i=i)dt},&&\int_{0}^\infty{e^{-\kappa t}\sup_{i \in \N^*}\left(\P(\t{X}_t^i=i)-\P(\t{X}_t^i=i-1)\right)dt}.
\end{align*}
The Markov process $(\t{X}_t)_{t\geq 0}$ which occurs is an alternative process and is not necessarily the same as the BDP$(\a,\b)$ with semigroup $(P_t)_{t\geq 0}$ appearing in the semigroup representation \eqref{eq:stein_sg_representation}. To our knowledge, this is new and makes the originality of our work.\\

\item While the detailed demonstrations of the forthcoming results are given in Section \ref{sect:proof_stein}, the scheme of proof is briefly explained here. Firstly, the argmax $f_i$ of the pointwise Stein factor
\begin{align*}
\sup_{f\in \F}{|\p^k g_f(i)|},&&i \in \N^*,\quad k \in \left\{0,1\right\},
\end{align*}
is obtained by resuming and generalizing results from \cite{brown2001stein}. Secondly, the function $f_i$ is plugged in the semigroup representation:
\begin{align*}
\p^k g_{f_i}(i)&=\int_0^\infty{\p^k \p^* P_t f_i\, dt}, &k\in \left\{0,1\right\}.
\end{align*}
The intertwining relations of Section \ref{sect:main} are then used to rewrite the term $\p^k \p^* P_t$.

This technique is already employed for Poisson approximation in some works, \cite{barbour1991pointprocess} and \cite{barbour2006stein} for example. In that context, the intertwining relation reads as:
\begin{align*}
\p P_t&=e^{-t}P_t\p,&&t\geq 0,
\end{align*}
where the semigroup $(P_t)_{t\geq 0}$ is the same on the left and on the right. The use of the intertwining relations permits to go beyond this case and to construct a universal method to derive Stein's factors.\\

\item For the sake of clarity, the present section only includes results on the uniform Stein factors. However, it can be seen in Section \ref{sect:proof_stein} that our upper bounds on the pointwise Stein factors are often sharp.\\

\item For the second order Stein factor, two sets of assumptions are used:
\begin{hypotheses}[Assumptions]\ 
\begin{itemize}
\item[$\mathbf{H_1}$:] The potential $V_1$ is non-increasing and non-negative, the potential $V_{1,*u}$ is bounded from below, and the sequence $u$ is bounded from below by a positive constant. In this case, we define $\sigma(1,*u)=\inf_\N V_{1,*u}$ and denote by $(X_{1,*u,t}^i)_{t \geq 0}$ the Markov process of generator $L_{1,*u}$ such that $X_{1,*u,0}^i=i$.
\item[$\mathbf{H_2}$:] The potential $V_1$ is a non-negative constant and the potential $V_{1,u}$ is bounded from below. In this case, set $\sigma(1,u)=\inf_\N V_{1,u}$ and call $(X_{1,u,t}^i)_{t \geq 0}$ the birth-death process of generator $L_{1,u}$ such that $X_{1,u,0}^i=i$.
\end{itemize}
\end{hypotheses}
This comes from the fact that the double intertwining relation is given by the main result Theorem \ref{theorem_intertwining_order2_hard} under $\mathbf{H_1}$ and by its analogous Theorem \ref{theorem_intertwining_order2_easy} under $\mathbf{H_2}$.

\item Stein's factors related to the different distances compare between each other through the inequalities:
\begin{align*}
\sup_{f=\1_{[0,m]},\,m\in\N}\|\p^k g_f\|_\infty &\leq \sup_{0\leq f\leq 1}\|\p^k g_f\|_\infty \leq \frac{1}{\inf_\N u}\sup_{f\in \text{Lip}(d_u)}\|\p^k g_f\|_\infty,& k\in \N. 
\end{align*}

\end{enumerate}

We now state the main results of this section, formulated for each distance of interest.\\

\textbf{Approximation in total variation distance.}

\begin{theorem}[First Stein's factor for bounded functions]
\label{thm:stein_factor_tv_order1}
Assume that $V_u$ is bounded from below by some positive constant $\sigma(u)$. Then, we have:
\begin{eqnarray*}
\sup_{0\leq f\leq 1}{\|g_f\|_\infty}\leq \int_0^\infty{e^{-\sigma(u)t}\sup_{i\in \N}\P(X_{u,t}^i=i)dt}.
\end{eqnarray*}
\end{theorem}
This theorem is applied to the negative binomial approximation in Proposition \ref{thm:neg_bin_bdd}.

\begin{theorem}[Second Stein's factor for bounded functions I]
\label{thm:stein_factor_tv_order2}
Under $\mathbf{H_1}$, 
\begin{eqnarray*}
\sup_{0\leq f \leq 1}\|\p g_f\|_\infty &\leq & 2 \int_0^\infty{e^{-\sigma(1,*u)t}\sup_{i \in \N^*}\P(X_{1,*u,t}^i=i)dt}.
\end{eqnarray*}
If the sequence is chosen to be $u=1$, we have:
\begin{eqnarray*}
\sup_{0\leq f \leq 1}\|\p g_f\|_\infty &\leq &  \int_0^\infty e^{-\sigma(1,*u)t}\sup_{i \in \N^*}\big( \P(X_{1,*u,t}^i=i)-\P(X_{1,*u,t}^i=i-1)\\
&&+\P(X_{1,*u,t}^i=i)-\P(X_{1,*u,t}^i=i+1)\big)\, dt.
\end{eqnarray*}
\end{theorem}

The analogue of Theorem \ref{thm:stein_factor_tv_order2} under the alternative set of hypotheses reads as:
\begin{theorem}[Second Stein's factor for bounded functions II]
\label{thm:stein_factor_tv_order2_bis}
Under $\mathbf{H_2}$, 
\begin{eqnarray*}
\sup_{0\leq f \leq 1}\|\p g_f\|_\infty &\leq & 2 \int_0^\infty{e^{-\sigma(1,u)t}\sup_{i \in \N}\P(X_{1,u,t}^i=i)dt}.\\
\end{eqnarray*}
If the sequence is chosen to be $u=1$, we have:
\begin{eqnarray*}
\sup_{0\leq f \leq 1}\|\p g_f\|_\infty &\leq &  \int_0^\infty e^{-\sigma(1,u)t}\sup_{i \in \N}\big(\P(X_{1,u,t}^i=i)-\P(X_{1,u,t}^i=i-1)\\
&&+\P(X_{1,u,t}^i=i)-\P(X_{1,u,t}^i=i+1)\big)\,dt.
\end{eqnarray*}
\end{theorem}

\begin{remark}[Alternative versions]
By the same techniques, it is possible to upper bound the quantities 
\begin{align*}
\sup_{0\leq f/u\leq 1} \|g_f\|_\infty,&&  \sup_{0\leq f/u\leq 1} \|\p_u g_f\|_\infty.
\end{align*}
It could be useful if one is interested in the approximation in $V$-norm (\cite{meyn1993tweedie}) rather than in total variation distance.
\end{remark}

\textbf{Approximation in Wasserstein distance.} 

\begin{theorem}[First Stein's factor for Lipschitz functions]
\label{thm:stein_factor_wasserstein_order1}
If $V_u$ is bounded from below by some positive constant $\sigma(u)$, then we have:
$$
\sup_{f\in \lip(d_u)} \Vert \Fw{g_f} /u \Vert_\infty  \leq \frac{1}{\sigma(u)}.
$$
Moreover, if $V_u$ is constant, then the preceding inequality is in fact an equality.
\end{theorem}

\begin{theorem}[Second Stein's factor for Lipschitz functions I]
\label{thm:stein_factor_wasserstein_order2}
Under $\mathbf{H_1}$, 
\begin{eqnarray*}
\sup_{f \in \text{Lip}(d_u)}\|\p_u g_f\|_\infty &\leq & \frac{1}{\sigma(1,*u)}\sup_{x \in \N^*}\left(1+\frac{u(x-1)}{u(x)}\right).
\end{eqnarray*}
If we assume that $u(x)=q^x$ on $\N$ with $q\geq 1$, then it stands that:
\begin{eqnarray*}
\sup_{f \in \text{Lip}(d_u)}\|\p_u g_f\|_\infty &\leq & \int_0^\infty e^{-\sigma(1,*u)t}\left(1-\frac{1}{q}+2\frac{1}{q}\sup_{i \in \N^*}\P(X_{1,*u,t}^i=i)\right)dt.
\end{eqnarray*}
\end{theorem}

An instance of Theorems \ref{thm:stein_factor_wasserstein_order1} and \ref{thm:stein_factor_wasserstein_order2} in the context of geometric approximation is given by Proposition \ref{thm:geometric_approximation_wasserstein}. The following theorem is the analogue of Theorem \ref{thm:stein_factor_wasserstein_order2} under the alternative set of hypotheses.
\begin{theorem}[Second Stein's factor for Lipschitz functions II]
\label{thm:stein_factor_wasserstein_order2_bis}
Under $\mathbf{H_2}$, 
\begin{eqnarray*}
\sup_{f \in \text{Lip}(d_u)}\|\frac{1}{u}\p\Fw{ g_f}\|_\infty &\leq & \frac{1}{\sigma(1,u)}\sup_{x \in \N}\left(1+\frac{u(x+1)}{u(x)}\right).
\end{eqnarray*}
If the sequence is chosen to be $u(x)=q^x$ on $\N$ with $q\geq 1$, then the following result holds:
\begin{eqnarray*}
\sup_{f \in \text{Lip}(d_u)}\|\frac{1}{u}\p\Fw{ g_f}\|_\infty &\leq & \int_0^\infty e^{-\sigma(1,u)t}\left(q-1+2\sup_{i \in \N}\P(X_{1,u,t}^i=i)\right)dt.\\
\end{eqnarray*}
\end{theorem}
As an illustration of this theorem, we derive Proposition \ref{thm:neg_bin_lip} in the case of negative binomial approximation.\\

\textbf{Approximation in Kolmogorov distance.}

The first theorem indicates that the inequality
\begin{align*}
\sup_{1_{[0,m]},\,m\in \N}{\|g_f\|_\infty} &\leq \sup_{0\leq f\leq 1}{\|g_f\|_\infty}
\end{align*}
is actually an equality. This comes from the fact that the function achieving the argmax of the pointwise factor for bounded functions is actually of the form $f=1_{[0,m]}$. As a consequence, our upper bounds for the first Stein factor are identical for the approximation in total variation and Kolmogorov distances.

\begin{theorem}[First Stein's factor for indicator functions]
\label{thm:stein_factor_kolmo_order1}
If $V_u$ is bounded from below and $\inf_\N{V_u}=\sigma(u)$, it stands that:
\begin{align*}
\sup_{1_{[0,m]},\,m\in \N}{\|g_f\|_\infty} &=\sup_{0\leq f\leq 1}{\|g_f\|_\infty}\leq \int_0^\infty{e^{-\sigma(u)t}\sup_{i\in \N}\P(X_{u,t}^i=i)dt}.
\end{align*}
\end{theorem}

The two following theorems deal with the second Stein factor under the two set of hypotheses.

\begin{theorem}[Second Stein's factor for indicator functions I]
\label{thm:stein_factor_kolmo_order2}
Under $\mathbf{H_1}$, 
\begin{eqnarray*}
\sup_{f=\1_{[0,m]},\,m\in \N} \|\p g_f\|_\infty &\leq & \int_0^\infty{e^{-\sigma(1,*u)t}\sup_{i\in \N} \P(X_{1,*u,t}^i=i)dt}.
\end{eqnarray*}
If the sequence is chosen to be $u=1$, we have:
\begin{eqnarray*}
\sup_{f=\1_{[0,m]},\,m\in \N} \|\p g_f\|_\infty &\leq &  \int_0^\infty e^{-\sigma(1,*u)t}\sup_{i \in \N^*}| \P(X_{1,*u,t}^i=i)-\P(X_{1,*u,t}^i=i-1)| dt.
\end{eqnarray*}
\end{theorem}

Comparing this second bound with the second bound obtained in Theorem \ref{thm:stein_factor_tv_order2} for the total variation approximation, one notices the fact that the second Stein factor for the total variation approximation involves the second derivative of the function $f_i(x)=\P(X_{1,*u,t}^i=x)$, whereas the second Stein factor for the Kolmogorov approximation involves only its first derivative.

Finally, we state the analogous of Theorem \ref{thm:stein_factor_kolmo_order2} under the alternative set of hypotheses.
\begin{theorem}[Second Stein's factor for indicator functions II]
\label{thm:stein_factor_kolmo_order2_bis}
Under $\mathbf{H_2}$, 
\begin{eqnarray*}
\sup_{f=\1_{[0,m]},\,m\in \N} \|\p g_f\|_\infty &\leq & \int_0^\infty{e^{-\sigma(1,u)t}\sup_{i\in \N} \P(X_{1,u,t}^i=i)dt}.
\end{eqnarray*}
If the sequence is chosen to be $u=1$, we have:
\begin{eqnarray*}
\sup_{0\leq f \leq 1}\|\p g_f\|_\infty &\leq &  \int_0^\infty{e^{-\sigma(1,u)t}\sup_{i \in \N}|\P(X_{1,u,t}^i=i)-\P(X_{1,u,t}^i=i-1)|dt}.
\end{eqnarray*}
\end{theorem}

\subsection{Stein's method and mixture of distributions}

\label{sect:mixture}

As another part of our work within the context of Stein's method, we present in the current section theoretical error bounds for the approximation of mixture of distributions. This section is independent from our study of Stein's factors contained in Section \ref{sect:stein-gen}. Results from both sections are combined in Section \ref{sect:exe} and applied to Poisson and geometric mixture approximation.

Let $\varphi$ be a non-negative function on $\N$ such that $\varphi(0)=0$. For $\lambda >0$, we denote by $\mathcal{I}_\varphi(\lambda)$ the probability distribution on $\N$ whose Stein's operator is
\begin{align*}
S_{\lambda} g(x) &=\lambda g(x+1) - \varphi(x)g(x), &x \in \N.
\end{align*}

By letting $\varphi$ vary, one finds back for $\mathcal{I}_\varphi(\lambda)$ every probability distribution supported on $\N$.
In particular, the choice $\varphi(x)=x$ gives the Poisson law and is studied in \cite{barbour1992poisson}. The choice $\varphi(x)=r+x$ and $\varphi(x)=1$ leads respectively to the binomial negative and geometric laws. A less classical example is $\varphi(x)=x^2$, for which $\mathcal{I}_\varphi(\lambda)$ is a distribution with pointwise probabilities $C_\lambda \lambda^x/(x!)^2$ for $x\in \N$ ($C_\lambda$ is the renormalizing constant).

The first theorem of this section reads as follows.

\begin{theorem}[Closeness of two $\mathcal{I}_\varphi(\lambda)$ distributions]
\label{th:bornesI}
Set $\lambda,\lambda'>0$. We have:
$$
d_\mathcal{F}(\mathcal{I}_\varphi(\lambda'),\mathcal{I}_\varphi(\lambda)) \leq |\lambda- \lambda'| \sup_{f\in \mathcal{F}} \Vert g_{\l,f}\Vert_\infty,
$$
where $g_f$ is the solution of the Stein's equation $S_\lambda g_f=f-\int{fd\mathcal{I}_\varphi(\lambda)}$. More generally, for any positive sequence $u$, if $X\sim \mathcal{I}_\varphi(\lambda)$ and  $X'\sim \mathcal{I}_\varphi(\lambda')$, then:
$$
d_\mathcal{F}(\mathcal{I}_\varphi(\lambda'),\mathcal{I}_\varphi(\lambda)) \leq |\lambda- \lambda'| \sup_{f\in \mathcal{F}} \Vert g_f/u\Vert_\infty \mathbb{E}[u(X'+1)].
$$
\end{theorem}

\begin{proof}
By the usual Stein error bound \eqref{eq:error_bound}, 
$$d_\mathcal{F}(\mathcal{I}_\varphi(\lambda'),\mathcal{I}_\varphi(\lambda))=\sup_{f\in \mathcal{F}}{\left|\int{fd\mathcal{I}_\varphi(\lambda')}-\int{fd\mathcal{I}_\varphi(\lambda)}\right|}=\underset{\substack{f\in \mathcal{F}}}{\sup}\left|\E[S_\lambda g_f(X')]\right|,$$
where $X'\sim \mathcal{I}_\varphi(\lambda')$. We know that $\E[S_{\lambda'} g_f(X')]=0$; this yields:
%\E[\varphi(X) g_f(X')] &= \lambda'  \E[g_f(X'+1)],\\
\begin{align*}
|\E[S_\lambda g_f(X')] |&= |\E\left[ (\lambda - \lambda') g_f(X'+1)\right]| = |(\lambda - \lambda') \E\left[ u(X'+1) g_f(X'+1)/u(X'+1)\right]| \\
&\leq |\lambda - \lambda'| \Vert g_f/u\Vert_\infty \E\left[ u(X'+1)\right].
\end{align*} 
\end{proof}
Note that the right hand side of both inequalities stated in Theorem \ref{th:bornesI} is not symmetric in $(\l,\l')$ due to the dependence of $g_f$ on $\lambda$ and one can slightly improve it by taking the minimum over the symmetrized form. The first inequality corresponds to the constant sequence $u=1$.

Let $W$ be a mixture of law $\mathcal{I}_\varphi(\lambda)$; namely there exists a random variable $\Lambda$ on $\R^+$ such that 
$$
\mathcal{L}( W \ | \ \Lambda) = \mathcal{I}_\varphi(\Lambda).
$$ 
(Recall that $\L(W)$ denotes the distribution of the random variable $W$.) A consequence of Theorem \ref{th:bornesI} is the following corollary.
\begin{corollary}[Biased approximation of mixed $\mathcal{I}_\varphi(\lambda)$ laws]
With the preceding notation, we have:
$$d_\mathcal{F}(\mathcal{L}(W),\mathcal{I}_\varphi(\lambda)) \leq \mathbb{E}[|\lambda- \Lambda|] \sup_{f\in \mathcal{F}} \Vert g_f\Vert_\infty.$$

\end{corollary}
\begin{proof}
Indeed,
$$d_\mathcal{F}(\mathcal{L}(W),\mathcal{I}_\varphi(\lambda)) \leq \mathbb{E}[d_\mathcal{F}(\mathcal{L}(W | \Lambda),\mathcal{I}_\varphi(\lambda))]  \leq \mathbb{E}[|\lambda- \Lambda|] \sup_{f\in \mathcal{F}} \Vert g_f\Vert_\infty.$$
\end{proof}
However, one actually has the following better bound using the mixture property of $W$:

\begin{theorem}[Unbiased approximation of mixed $\mathcal{I}_\varphi(\lambda)$ distributions]
\label{th:mixture}
For every positive sequence $u$, letting $\lambda= E[\Lambda]$, we have:
\begin{align*}
d_\mathcal{F}(\mathcal{L}(W),\mathcal{I}_\varphi(\lambda)) 
&\leq \sup_{ f \in \mathcal{F}} \Vert \partial_u g_f \Vert_\infty\, \underset{f \in \lip(d_{\fw{u}})}{\sup\quad\quad} \Vert g_{f} \Vert_\infty   \, \textrm{Var}(\Lambda) .
\end{align*}
More generally, the following upper bound holds for all positive sequences $u,v$:
\begin{align*}
d_\mathcal{F}(\mathcal{L}(W),\mathcal{I}_\varphi(\lambda)) 
&\leq \sup_{ f \in \mathcal{F}} \Vert \partial_u g_f \Vert_\infty\, \underset{r\in \lip(d_{\overset{\rightarrow}{u}})}{\sup\quad\quad}\,\|g_r/v\|_\infty   \, \E[ |\lambda- \Lambda|^2 \mathbb{E}[v(W+1)\  | \ \Lambda]].
\end{align*}
\end{theorem}

\begin{proof}[Proof of Theorem \ref{th:mixture}]
For every real-valued function $g$ on $\N$, $\E[S_\Lambda(g)(W)|\Lambda]=0$. Hence, by taking $g=g_f$ the solution to Stein's equation associated with any fixed function $f:\N\rightarrow\R$, 
\begin{eqnarray*}
\E[S_\lambda g_f (W)]&=&\E[\E[(S_\lambda -S_\Lambda)g_f(W)|\Lambda]]=\E[(\lambda-\Lambda)g_f(W+1)]\\
&=&\E[(\lambda-\Lambda)(g_f(W+1)-g_f(Z+1))],
\end{eqnarray*}
where $Z\sim \mathcal{I}_\varphi(\lambda)$. For two random variables $Z, Z'$ on $\N$, by the Kantorovich-Rubinstein theorem recalled at the beginning of this section, 
\begin{eqnarray*}
\left|\E[g(Z'+1)-g(Z+1)]\right| & \leq &\|\p_u g\|_\infty W_{d_u}(\mathcal{L}(Z'+1),\mathcal{L}(Z+1))=\|\p_u g\|_\infty \inf \E[d_{u}(Z'+1,Z+1)]\\
&=&\|\p_u g\|_\infty W_{d_{\overset{\rightarrow}{u}}}(\mathcal{L}(Z'),\mathcal{L}(Z)),
\end{eqnarray*}
where the infimum is taken on the set of couplings with first marginal $\mathcal{L}(Z)$ and second marginal $\mathcal{L}(Z')$.
Now, by Theorem \ref{th:bornesI},
\begin{eqnarray*}
\E[S_\lambda g_f (W)]&\leq & \|\p_u g\|_\infty \E[|\lambda-\Lambda|\E[W_{d_{\overset{\rightarrow}{u}}}(\mathcal{L}(W),\mathcal{L}(Z))|\Lambda]\\
&\leq & \|\p_u g\|_\infty  \sup_{r\in \lip(d_{\overset{\rightarrow}{u}})}\|g_r/v\|_\infty \E[(\lambda-\Lambda)^2 \E[v(W+1)| \Lambda ]].
\end{eqnarray*}
As in Theorem \ref{th:bornesI}, the first inequality is an instance of the second one in the case $v=1$.
\end{proof}

\begin{remark}[Alternative bound via coupling]
\label{rq:mixture_unbiased_bis}
In the previous proof, we used Theorem \ref{th:bornesI} in order to bound $W_{d_{\overset{\rightarrow}{u}}}(\mathcal{I}_\varphi(\Lambda),  \mathcal{I}_\varphi(\lambda))$. It is also possible to bound this distance via another method (for instance a coupling argument) instead of using a bound on Stein's solution.
\end{remark}

\section{Examples}

\label{sect:exe}
In this section we illustrate our results on some examples. The classical examples of the M/M/$\infty$ and M/M/$1$ process come from the queueing theory. We also apply the results to the Galton-Watson process with immigration. Other explicit examples of birth-death processes for which a "good choice" of sequence $u$ is known are given in \cite[Table 9.1 p. 351]{chen2004markov} and in \cite{chen1996estimation}. For the sake of conciseness we defer the proof of Lemma \ref{lem:majo_ptw_poisson} about the pointwise probabilities of the M/M/$\infty$ queue to Section \ref{sect:proof_expl_mminfty}.

\subsection{The M/M/$\infty$ process and the Poisson approximation}
\label{sect:mminfty}

Let $(X_t)_{t\geq 0}$ be a BDP with constant birth death $\lambda$ and linear death rate $x\mapsto x$. Its invariant measure is the Poisson law $\mathcal{P}_\lambda$. Let us set $u=v=1$ on $\N$. By application of Theorem \ref{theorem_intertwining_order1} we find that $V_1=1$ and that $(P_{1,t})_{t\geq 0}=(P_{t})_{t\geq 0}$. Applying Theorem \ref{theorem_intertwining_order2_easy} (or re-applying Theorem \ref{theorem_intertwining_order1}) yields $V_{1,1}=2$ and $(P_{1,1,t})_{t\geq 0}=(P_{t})_{t\geq 0}$. By a straightforward induction, for all positive or bounded functions $f:\N\rightarrow\R$,
\begin{align}
\label{formula_intertwining_poisson}
\p^k P_t f&=e^{-kt}P_t  \p^k f, &t\geq 0,\quad k \in \N.
\end{align}

Combined with Theorem \ref{theorem_contraction_semigroup_zolotarev_distance} and Remark \ref{rq:zolo-gen}, the equation \eqref{formula_intertwining_poisson} implies the following contraction in Zolotarev-type distance: for all $\mu \in \mathcal{P}$,
\begin{align*}
\underset{\|\p^k f\|_\infty\leq 1}{\sup\quad\;}\left|\mu(P_t f)-\mathcal{P}_\lambda(f)\right| &\leq e^{-kt}\underset{\|\p^k f\|_\infty\leq 1}{\sup\quad\;}\left|\mu(f)-\mathcal{P}_\lambda(f)\right|, &k\in \N^*.
\end{align*}

Formula \eqref{formula_intertwining_poisson} is already known and often proved using Mehler's formula which reads, for any bounded function $f$, as: 
\begin{align}
\label{eq:mehler-poisson}
P_tf(x) &= \mathbb{E}[f(X^0_t + B_t)],& x\in \N,\quad t\geq 0,
\end{align}
where $(X^0_t)_{t\geq 0}$ is a M/M/$\infty$ process starting from $0$ and $B_t$ is an independent random variable distributed as a binomial random variable with parameters $(x,e^{-t})$. It is also known that $X^0_t$ is distributed as a Poisson distribution with parameter $\lambda(1-e^{-t})$ at every time $t\geq 0$.
Conversely, the proof of the formula (\ref{eq:mehler-poisson}) can be deduced from Theorem \ref{theorem_intertwining_order1} with similar (but simpler) arguments than those developed in Lemma \ref{lemma_mehler_negbin} below.\\

We now turn to the subject of Poisson approximation and the associated Stein factors. Let $g_f$ be the solution to Stein's equation (\ref{equation_stein}) with Stein's operator $Sf(x)=\lambda f(x+1)-xf(x)$. The target measure is the Poisson distribution $\mathcal{P}_\lambda$. The following lemma allows to estimate from above the pointwise probabilities of the process $(X_t)_{t\geq 0}$.

\begin{lemma}[Upper bounds of the instantaneous probabilites of the M/M/$\infty$ queue]
\label{lem:majo_ptw_mminfty}
Let $(X_t^x)_{t\geq 0}$ be a BDP$(\l,x)_{x\in \N}$. For all $x\in \N$ and $t\geq 0$,
\begin{align*}
\sup_{x \in \N}\,\P(X_t^x=x) & \leq 1 \wedge \frac{c}{\sqrt{\l (1-e^{-t})}}, &c&= \frac{1}{\sqrt{2e}},\\
\sup_{x \in \N^*}\,\left|\P(X_t^x=x)-\P(X_t^x=x-1)\right| & \leq 1 \wedge \frac{C}{\l (1-e^{-t})},&C&=\frac{1}{\sqrt{2\pi}}e^{\frac{1}{\sqrt{2}}} \leq 1.
\end{align*}
\end{lemma}
The first upper bound is very classical, it derives from Mehler's formula \eqref{eq:mehler-poisson} and an upper bound on the pointwise probabilitites of the Poisson distribution (\cite[Proposition A.2.7]{barbour1992poisson}). The second one is new and is proved in Section \ref{sect:proof_expl_mminfty}, since it is rather technical and can be omitted at first reading.\\

By applying Theorems \ref{thm:stein_factor_tv_order1}, \ref{thm:stein_factor_wasserstein_order1}, \ref{thm:stein_factor_wasserstein_order2_bis} jointly with the first bound of Lemma \ref{lem:majo_ptw_mminfty}, one finds back (and by the same techniques) the following upper bounds (\cite{barbour1991pointprocess}, \cite{barbour2006stein}):
\begin{align*}
 \sup_{0 \leq f\leq 1} \|g_f \|_\infty&\leq 1\wedge \sqrt{\frac{2}{\l e}},& \sup_{f \in \lip(d_1)} \|g_f \|_\infty&=1,  & \sup_{f \in \lip(d_1)} \|\p g_f \|_\infty &\leq 1 \wedge \frac{8}{3\sqrt{2e\lambda}}.
\end{align*}
Of course, one may want to derive other known Stein's factors for Poisson approximation by our techniques, as for instance the second Stein factor for approximation in the total variation distance with rate $1\wedge (1/\l)$ (\cite{barbour1992poisson}). However, when applying Theorem \ref{thm:stein_factor_tv_order2_bis} 
%or \ref{thm:stein_factor_kolmo_order2_bis} 
with the second bound of Lemma \ref{lem:majo_ptw_mminfty}, the non-integrability in $0$ of the term $1/(1-e^{-t})$ leads to sub-optimal results (namely, after some careful computations, we recover the known rate, up to a multiplicative factor $\log \l$).\\

Let us now combine the Stein bounds with our results on the mixture of distributions. If $\varphi(x)=x$ then $\mathcal{I}_\varphi (\lambda)=\mathcal{P}_\lambda$. In particular, Theorem \ref{th:bornesI} and the preceding bounds give
\begin{align*}
d_{\mathrm{TV}}(\mathcal{P}(\lambda), \mathcal{P}(\lambda')) &\leq \frac{1}{1 \wedge \sqrt{\lambda \vee \lambda'}}|\lambda - \lambda'|,&W(\mathcal{P}_\lambda, \mathcal{P}_{\lambda'}) &\leq |\lambda - \lambda'|.
\end{align*}

The first bound is (almost) the result of \cite[Theorem 1.C p. 12]{barbour1992poisson}. The second one is in fact an equality and can also be proved via a coupling approach (\cite{lindvall2002coupling}). Theorem \ref{th:mixture} yields
\begin{align*}
W(\mathcal{L}(W),\mathcal{I}_\varphi(\lambda)) &\leq \left(1 \wedge \frac{8}{3\sqrt{2e\lambda}} \right) \textrm{Var}(\Lambda) , &d_{\mathrm{TV}}(\mathcal{L}(W),\mathcal{I}_\varphi(\lambda))& \leq  \frac{1}{\lambda}    \textrm{Var}(\Lambda).
\end{align*}
 
While the second bound is exactly the same as in \cite[Theorem 1.C p. 12]{barbour1992poisson}, the bound in Wasserstein distance seems to be new. Let us see an instance of it. We denote by NB$(r,p)$ the negative binomial distribution of parameters $(r,p)$, i.e.,
\begin{align*}
\text{NB}(r,p)(x)&=\frac{\Gamma(r+x)}{\Gamma(r)x!}(1-p)^r p^x,&x\in \N,
\end{align*}
where $\Gamma$ denotes the usual $\Gamma$ function. The negative binomial law is a mixed Poisson distribution with $\Lambda$ distributed as a Gamma law with parameters $r$ and $\frac{1-p}{p}$. Consequently, we obtain
$$
W\left(\text{NB}(r,p),\mathcal{P}_{r(1-p)/p}\right) 
\leq \frac{8}{3\sqrt{2e}} \sqrt{\frac{r(1-p)}{p}}\frac{(1-p)}{p},
$$
which is the upper bound announced in the introduction.
%, \quad d_{\mathrm{TV}}\left(\text{NB}(r,p),\mathcal{P}\left(r(1-p)/p\right)\right) \leq  \frac{1-p}{p}. 
A similar approximation in total variation distance holds. Although the convergence of the binomial negative distribution towards a Poisson law in the regime $p\to 1$, $r\to \infty$ and $r(1-p)/p\rightarrow c$ for a positive constant $c$ is a well-known fact, the preceding upper bound seems to be the first attempt to quantify this convergence. \\

\subsection{The GWI process and the negative binomial approximation}
\label{sec:GWI}

We consider the BDP with rates $\a(x)=p(r+x)$, $\b(x)=x$ on $\N$ with $r>0$ and $0<p<1$. The coefficient $pr$ can be interpreted as a rate of immigration, while the birth rate per capita is $p$ and the death rate per capita is $1$. Without the immigration procedure, this is a Galton-Watson process whose individuals have only one descendant (or simply a linear birth-death process).
The invariant measure of this process is the negative binomial distribution $\text{NB}(r,p)$ just defined. Remark that for the particular choice $r=1$ it is nothing else than the geometric law of parameter $p$. If $X$ is a $\text{NB}(r,p)$  random variable then $X+r$ follows the so-called Pascal distribution; it represents the number of successes in a sequence of independent and identically distributed Bernoulli trials (with parameter $p$) before $r$ failures when $r$ is a positive integer.

Let us take $u=v=1$ on $\N$. Theorem \ref{theorem_intertwining_order1} shows that:
\begin{align*}
\p P_t&=P_{1,t}^{V_1},&t\geq 0,
\end{align*}
where $(P_{1,t})_{t\geq 0}$ is a birth-death process with rates defined as
\begin{align*}
\a_1(x)&=p(r+1+x), & \b_1(x)&=x , & x\in \N.
\end{align*}
It is again a Galton-Watson process with immigration. The birth and death rates are unchanged and the immigration rate is increased by $p$. The potential $V_1$ is constant and takes the value $V_1=1-p$. By Theorem \ref{theorem_intertwining_order2_easy}, we find that $(P_t)_{t\geq 0}$ and $\p^2$ are intertwined via the Feynman-Kac semigroup composed of a birth-death semigroup with rates $(\a_{1,1},\b_{1,1})$ and of potential $V_{1,1}$, with:
\begin{align*}
\a_{1,1}(x)&=p(r+2+x), & \b_{1,1}(x)&=x, &  V_{1,1}(x)&=2(1-p), & x\in \N.
\end{align*}

Let us call $(P_{k,t})_{t\geq0}$ the semigroup associated to a BDP with rates $(p(r+k+x),x)$ on $\N$. By a straightforward induction, for all positive or bounded functions $f:\N\rightarrow\R$, the following intertwining relation holds:

\begin{align}
\label{formula_intertwining_negative_binomial}
\p^k P_tf &=e^{-(1-p)kt}P_{k,t}\, \p^k f, & t \geq 0, \quad k\in \N.
\end{align}

As indicated in Remark \ref{rq:zolo-gen}, the previous equality gives the following improvement of Theorem \ref{theorem_contraction_semigroup_zolotarev_distance}: for every $\mu \in \mathcal{P}$,
$$
\sup_{\Vert \partial^k f \Vert_\infty \leq 1} \left|\mu (P_t f) - \pi(f)\right| \leq e^{-(1-p)kt}\sup_{\Vert \partial^k f \Vert_\infty \leq 1} \left|\mu(f)-\pi(f)\right|.
$$
Another consequence of the formula (\ref{formula_intertwining_negative_binomial}) is the invariance of polynomials under the action of $(P_t)_{t\geq 0}$: if $Q$ is a polynomial of degree $k$, then for all $t\geq 0$, $\p^k P_t Q$ is constant, hence $P_t Q$ is still a polynomial of degree $k$. This property also holds for the M/M/$\infty$ process. \\

Intertwining relations can be seen in certain cases as consequences of Mehler-type formulas. Here, conversely, we are able to derive a Mehler-type formula from the first order intertwining relation. To our knowldedge, this formula is new, though another Mehler-type formula is proved in \cite{barbour2015stein}.
\begin{lemma}[A Mehler's formula for the Galton-Watson process with immigration]
\label{lemma_mehler_negbin}
Set $0<\nolinebreak p< \nolinebreak 1$, $s>0$ and $q=1-p$. For all $x\in \N$ let $(Y_t^x)_{t\geq 0}$ be a birth-death process starting at $x$ and with rates $(p(s+k),k)_{k\in \N}$. Let $W_t$ be a random variable following the Poisson distribution $\mathcal{P}(p(1-e^{-t}))$ and define the sequence $(w(k))_{k\in \N}$ as 
$$w(0)=1-e^{-qt}\P(W_t=0)\quad\text{ and }\quad\forall k \in \N^*,\quad w(k)=e^{-qt}(\P(W_t=k-1)-\P(W_t=k)).$$
For all $t\geq 0$, let the random variables $(Z_{i,t})_{i\in \N}$ be independent, identically distributed and independent of $Y_t^0$, with distribution given by the pointwise probabilities $(w(k))_{k\in \N}$. Then we have the equality in law
$$Y_t^x\,=\,Y_t^0+\sum_{i=1}^x{Z_{i,t}}.$$
\end{lemma}

\begin{proof}
This proof is a corollary of the intertwining formula (\ref{formula_intertwining_negative_binomial}) for $k=1$. Indeed, for every bounded real-valued function on $\N$, Theorem \ref{theorem_intertwining_order1} implies that 
\begin{align*}
\E[f(Y_t^{x+1})]&=\E[f(Y_t^{x})]+e^{-qt}\E[f(\t{Y}_t^x+1)-f(\t{Y}_t^x)],& x\in \N, \quad t\geq 0,
\end{align*}
where $(\t{Y}_t^x)_{t\geq 0}$ is a BDP $(p(s+1+k),k)_{k\in \N}$. We notice that $(\t{Y}_t^x)_{t\geq 0}=({Y}_t^x+W_t)_{t\geq 0}$, where $(W_t)_{t\geq 0}$ is a birth-death process independent of $(Y_t^x)_{t\geq 0}$ with rates $(p,k)_{k\in \N}$ such that $W_0=0$. The process $(W_t)_{t\geq 0}$ is a M/M/$\infty$ queue starting from $0$ at time $0$. It is distributed as a Poisson law $\mathcal{P}_{\lambda_t}$, $\lambda_t=p(1-e^{-t})$ at all times $t\geq 0$. We use below the observation that as $\lambda_t<1$ for all $t\geq 0$, the sequence $(\P(W_t=k))_{k\in \N}$ is non-increasing on $\N$. We have:
\begin{eqnarray*}
\E[f(Y_t^{x+1})]&=&\E[f(Y_t^{x})]+e^{-qt}\sum_{k=0}^\infty{\P(W_t=k)\E[f(Y_t^x+k+1)-f(Y_t^x+k)]}\\
&=&(1-e^{-qt}\P(W_t=0))\E[f(Y_t^{x})]+e^{-qt}\sum_{k=1}^\infty{(\P(W_t=k-1)-\P(W_t=k))\E[f(Y_t^{x}+k)]}\\
&=&\sum_{k=0}^\infty{w(k)\E[f(Y_t^{x}+k)]},
\end{eqnarray*}
where the sequence $(w(k))_{k\in \N}$ is defined in the statement of the lemma. It is easy to check that $\sum_{k=0}^\infty{w(k)}=1$, and that the sequence $(w(k))_{k\in \N}$ is non-negative thanks to the observation above. For all $t\geq 0$, we define a random variable $S_t$ such that $S_t$ is independent of $(Y_t^x)_{t\geq0}$ and that for all non-negative integer $\P(S_t=k)=w(k)$. This yields the equality in law $Y_t^{x+1}=\nolinebreak Y_t^x+\nolinebreak S_t$. The lemma follows by induction.
\end{proof}

Let us turn to the study of Stein's factors associated to the negative binomial approximation. We recall a lemma from \cite{barbour2015stein} on the instantaneous probabilities of a Galton-Watson process with immigration, and give an outline of the proof for the sake of completeness. One could also use Lemma \ref{lemma_mehler_negbin} jointly with Lemma \ref{lem:majo_ptw_poisson} to upper-bound these instantaneous probabilities, but the majoration obtained by doing so does not reveal practical to use. 

\begin{lemma}[Upper bound of the instantaneous probabilities of a GWI process]
\label{lem:majo_ptw_gwi}
Set $(X_t^x)_{t \geq 0}$ be a BDP$(p(r+k),k)_{k \in \N}$. Then:
\begin{align*}
\sup_{x\in \N}{\P(X_t^x=x)}&\leq 1 \wedge \frac{1}{\sqrt{2e}} \left(\frac{1-p}{\,p\,(1-e^{-(1-p)t})}\right)^{1/2}\frac{K(r)}{\sqrt{r}},& t\geq 0,\quad  p\in (0,1),\quad  r>\frac{1}{2},
\end{align*}
with $K(r)=\sqrt{r}\Gamma(r-1/2)/\Gamma(r)$. 
\end{lemma}

\begin{proof}
By Lemma \ref{lemma_mehler_negbin},
\begin{align*}
\sup_{x \in \N}{\P(X_{t}^x=x)} &\leq  \sup_{x \in \N}{\P(X_{t}^0=x)},&t\geq 0.
\end{align*}
By a result of \cite{kendall1948modes}, cited as Lemma 2.2 in \cite{barbour2015stein}, it is known that for all $t\geq 0$, $X_{t}^0$ is distributed as a negative binomial distribution of parameters $(r, \theta_t(p))$, with 
$$\theta_t(p)=1-\frac{1-p}{1-pe^{-(1-p)t}}.$$ 
Now \cite{phillips1996stochastic} shows that when $X$ is distributed as a negative binomial distribution with parameters $(r,\theta)$, and if $r>\frac{1}{2}$, then 
$$\sup_{k\in \N}\P(X=k) \leq \frac{1}{\sqrt{2e}}\sqrt{\frac{1-\theta}{\theta}}\frac{K(r)}{\sqrt{r}},$$
which achieves the proof.
\end{proof}

For the Stein factor associated with Lipschitz function, Theorem \ref{thm:stein_factor_wasserstein_order1} and equation \eqref{formula_intertwining_negative_binomial} yield
$$\sup_{f \in \text{Lip}(d)}\| g_f\|_\infty=\frac{1}{\sigma(1)}=\frac{1}{1-p},$$
recovering \cite[Theorem 1.1, equation (1.3)]{barbour2015stein}.

The following proposition on the second Stein factor associated to Lipschitz function improves on the known upper bounds.

\begin{proposition}[Estimation of the second Stein's factor for Lipschitz function and NB-approximation]
\label{thm:neg_bin_lip}
Let $r>0$ and $0<p<1$.
For a real-valued function $f$ on $\N$, let $g_f$ be the (centered) solution to Stein's equation
$$p(r+x)\,\p g_f(x)\,+\,x\,\p^*g_f(x)=f(x)-\int{fd\text{NB}(r,p)}, \quad x\in \N.$$
Then,
\begin{eqnarray*}
\sup_{f \in \text{Lip}(d)}\|\p g_f\|_\infty &\leq&   \min\left\{\frac{1}{1-p},\frac{D}{\sqrt{(r+2)p(1-p)}}\right\},\quad \quad D=2\frac{\sqrt{\pi}}{3\sqrt{ e}}\simeq 0.72.
\end{eqnarray*}
\end{proposition}

\begin{proof}
By application of Theorem \ref{theorem_stein_factor_wasserstein} and formula \eqref{formula_intertwining_negative_binomial}, we find that 
$$\sup_{f \in \text{Lip}(d)}\|\p g_f\|_\infty =  2\int_{0}^\infty{e^{-2(1-p)t}\,\sup_{i \in \N}\,{\P(X_{1,1,t}^i=i)}dt},$$ 
where $(X_{1,1,t}^i)_{t\geq0}$ is a BDP $(p(r+2+x),x)_{x\in \N}$.
Applying Lemma \ref{lem:majo_ptw_gwi}, 
\begin{align*}
\sup_{f \in \text{Lip}(d)}\|\p g_f\|_\infty &\leq 2 \int_{0}^\infty{e^{-2(1-p)t} dt} \wedge 2 \sqrt{\frac{1-p}{p}}\frac{K(r+2)}{\sqrt{(2e)(r+2)}} \int_0^\infty{\frac{e^{-2(1-p)t}}{\sqrt{1-e^{-(1-p)t}}}dt}.
\end{align*}
The function $K$ is decreasing on $(\frac{1}{2},\infty)$, hence $K(r+2)\leq K(2)$. (The function $K$ is bounded from below by a positive constant on $(\frac{1}{2},\infty)$, hence by writing this majoration we do not lose the rate in $r$.) Furthermore, 
$$\int_0^\infty{\frac{e^{-2(1-p)t}}{\sqrt{1-e^{-(1-p)t}}}dt}=\frac{4}{3(1-p)}.$$ 
Finally,
\begin{eqnarray*}
\sup_{f \in \text{Lip}(d)}\|\p g_f\|_\infty &\leq&   \min\left\{\frac{1}{1-p},\frac{D}{\sqrt{(r+2)p(1-p)}}\right\} ,
\end{eqnarray*}
with $D=\frac{8K(2)}{3\sqrt{2e}}=\frac{4\Gamma(3/2)}{3\sqrt{e}}=\frac{2\sqrt{\pi}}{3\sqrt{ e}}.$ 
\end{proof}

The Proposition \ref{thm:neg_bin_lip} might be compared to \cite[Theorem 1.1, equation (1.4)]{barbour2015stein}, which states the inequality
\begin{equation}
\label{eq:negbin_approx_barbour}
\sup_{f \in \text{Lip}(d)}\|\p g_f\|_\infty \leq \min\left\{ \frac{2}{1-p}, \frac{1+p}{(1-p)^2},\frac{1.5}{\sqrt{rp(1-p)^3}}\right\}.
\end{equation}
We observe that:
\begin{itemize}
\item The numerical constant in front of $1/(1-p)$ is improved.
\item As $\frac{D}{\sqrt{(r+2)p(1-p)}}\leq \frac{0.8 }{\sqrt{rp(1-p)}}$ and $0.8\leq \frac{1.5}{1-p}$, we have:
$$\frac{D}{\sqrt{(r+2)p(1-p)}} \leq  \frac{1.5}{\sqrt{rp(1-p)^3}}.$$
\end{itemize}
Note that the proofs are similar up to the formula
$$\sup_{f \in \text{Lip}(d)}\ |\p g_f(i)|=-\int_0^\infty{\p \p^* P_t \1_i dt}.$$
We then apply the second order intertwining formula, whereas \cite{barbour2015stein} use another technique. In both cases, a bound of the type $\sup_i \mathbb{P}(Y_t^i=i)$ is needed, but not for the same process $(Y_t)_{t\geq 0}$.\\

For the Stein factor associated to bounded functions, at the order $1$ we find the following result.
\begin{proposition}[Estimation of the first Stein factor for bounded functions and NB-approximation]
\label{thm:neg_bin_bdd}
With the same assumptions as in Theorem \ref{thm:neg_bin_lip}, we have:
\begin{eqnarray*}
\sup_{0\leq f \leq 1}\| g_f\|_\infty &\leq&   \frac{1}{1-p}\wedge \frac{\sqrt{\pi}}{\sqrt{(r+1)p(1-p)}}.
\end{eqnarray*}
\end{proposition}

We do not detail the proof which is very similar to the one of Proposition \ref{thm:neg_bin_lip}.

This result improves on a result of \cite[Lemma 3]{brown1999negative} which states
$$\sup_{0\leq f \leq 1}\| g_f\|_\infty \leq \frac{1}{p \vee (1-p)\1_{r\geq 1} } .$$

We do not develop the case of the second Stein factor of bounded functions, where the upper bound given by Theorem \ref{thm:stein_factor_tv_order2} recovers the simple inequality
$$\sup_{0 \leq f \leq 1}\|\p g_f\|_\infty \leq  \sup_{f \in \lip(d_1)}\|\p g_f\|_\infty.$$
Results about this factor can be found in \cite[Theorem 2.10]{brown2001stein}, in \cite[example 2.12]{eichelsbacher2008stein} for the case $r=1$, and in \cite[Lemma 5]{brown1999negative}.\\
% The second Stein factors associated to half-line indicator functions have been studied in \cite{teerapabolarn2013improved}.\\

%
%As a consequence, we also derive a bound for the second Stein's factor for bounded functions (see Remark \ref{rq:vtot_compared_w}):
%\begin{eqnarray*}
%\sup_{0\leq f\leq 1}\|\p g_f\|_\infty &\leq&   2\min\left\{\frac{1}{1-p},\frac{C}{\sqrt{(r+2)p(1-p)}}\right\}
%\end{eqnarray*}
%which is to compare to \cite[Theorem 2.10]{brown2001stein} which yields in the present case 
%\begin{eqnarray*}
%\sup_{0\leq f\leq 1}\|\p g_f\|_\infty &\leq&  \min\left\{1, \frac{1}{pr}\right\},
%\end{eqnarray*}
%to \cite[example 2.12]{eichelsbacher2008stein} where it is shown in the special case $r=1$ ($i.e.$ Geometric law) that 
%\begin{eqnarray*}
%\sup_{0\leq f\leq 1}\|\p g_f\|_\infty &\leq&  \min\left\{1, 1+p\right\},
%\end{eqnarray*}
%and to \cite[Lemma 5]{brown1999negative}
%\begin{eqnarray*}
%\sup_{0\leq f \leq 1} \Vert \partial g_f \Vert \leq  \frac{1- p^r}{r(1-p)}.
%\end{eqnarray*}
%For the sake of completeness, let us also observe that the second Stein's factors associated to indicator functions have also been studied :
%\begin{eqnarray*}
%\sup_{f=\1_m,\, m \in \N} \Vert \partial g_f \Vert \leq  \frac{1- p^{r+1}}{(r+1)(1-p)}, \quad \sup_{f=\1_{[0,m]},\, m \in \N} \Vert \partial g_f \Vert \leq  \frac{1}{2},\quad r=1.
%\end{eqnarray*}
%The first upper bound is stated in \cite{jaioun2015uniform}, the second one in \cite{teerapabolarn2013improved}.\\	

If $\varphi: x \mapsto r+x$, $r \in \N$ and $\lambda \in (0,1)$ then $\mathcal{I}_\varphi(\lambda)=$NB$(r,\lambda)$. The variable $W+r$ then represents the number of trials that are necessary to obtain $r$ successes in a Bernoulli experiment with a random probability of gain.\\

To conclude this section, we observe that the Stein operator associated to a probability measure is not unique, and that resulting Stein's factors depend on the choice of the operator. When $r=1$, we recover the geometric law as the invariant distribution, similarly to the forthcoming example. This is the choice of \cite{eichelsbacher2008stein} to study the geometric distribution. In the next section we choose another Stein's operator.

\subsection{The M/M/$1$ process and the geometric approximation}
\label{sect:exemplesMM1}

Let $(X_t^x)_{t\geq 0}$ be a BDP$(\a,\b)$ with rates $\a(x)=\a,\,\b(x)=\b\1_{x\in \N^*}$ on $\N$. We suppose that $\rho:=\frac{\a}{\b}<1$. We denote by $(P_t)_{t\geq 0}$ the associated semigroup. Its invariant distribution is the geometric law $\mathcal{G}(\rho)$ with pointwise probabilities $p(k)=(1-\rho)\rho^k$ for $k \in \N$. Notice that this is the definition of the geometric law with support $\N$ and not $\N^*$. Let us choose $u(x)=r^x,\,v(x)=q^x$ for $x\in \N$ with $r>0, \,q\geq 1$. Theorem \ref{theorem_intertwining_order1} gives rise to a Feynman-Kac semigroup composed of a birth-death semigroup $(P_{u,t})_{t\geq 0}$ with rates $(\a_u,\b_u)$ and a potential $V_u$, which are defined as
\begin{align*}
\a_u(x)&=r\a, & \b_u(x)&=\frac{1}{r}\b,& V_u(x)&=(1-r)\a+\left(1-\frac{1}{r}\1_{x\in \N^*}\right)\b,&x\in \N.
\end{align*}
$$$$
The semigroup $(P_{u,t}^{V_u})_{t\geq 0}$ is still a semigroup associated to a M/M/$1$ queue, only with modified rates. The potential $V_u$, while non-constant, is non-increasing on $\N$. By Theorem \ref{theorem_intertwining_order2_hard}, we find a Feynman-Kac semigroup $(P_{u,*v,t}^{V_{u,*v}})_{t\geq 0}$ where $(P_{u,*v,t})_{t\geq 0}$ is again a semigroup corresponding to a M/M/$1$ queue. The rates and potential are defined on $\N$ as 
\begin{align*}
\a_{u,*v}(x)&=q r\a, \quad \b_{u,*v}(x)=\frac{1}{q r}\b\1_{x\in \N^*}, & x\in \N,&&&\\
V_{u,*v}(x)&=(1-qr)\a+\left(1-\frac{1}{qr}\right)\b,& x\in \N^*&,& V_{u,*v}(0)&=\a -(1+q)r\a +\b.
\end{align*}

Remark that, in contrast with the general case of Theorem \ref{theorem_intertwining_order2_hard}, the semigroup $(P_{u,*v,t})_{t\geq 0}$ is again a birth-death semigroup. This is due to the fact that $V_u$ is constant on $\N^*$. The potential $V_u$ is not constant on $\N$, which prevents us to apply Theorem \ref{theorem_intertwining_order2_easy}, but it is almost constant which explains heuristically why we find again a birth-death process when applying Theorem \ref{theorem_intertwining_order2_hard}.\\

Set $\sigma(u,*v)=\inf_{x\in \N}{V_{u,*v}(x)}=\min(V_{u,*v}(0), V_{u,*v}(1))$. A few calculations show that
$$\max \left\{\sigma(u,*v)\,|\,u(x)=r^x, \,v(x)=q^x,\, r>0, \, q\geq 1 \right\}=(\sqrt{\beta}-\sqrt{\alpha})^2,$$
and the $\arg \max$ is realized for all $r\leq \sqrt{\b/\a}=\rho^{-1}$ and $q=\rho^{-1}/r$. This means that there is a range of choice for the parameters $(r,q)$ allowing to recover the spectral gap $(\sqrt{\b}-\sqrt{\a})^2$ of the process in the convergence of Theorem \ref{theorem_contraction_semigroup_zolotarev_distance}. However, contrary to the two preceding examples, notice that the second order intertwining does not allow to improve on the spectral gap and that the rate of convergence in the distance $\zeta_{u,*v}$ is the same as the rate of convergence in the Wasserstein distance $W_{d_u}$ for the best choices of $u,v$.\\

This example is maybe the most important because, in contrast with the two previous processes, the M/M/$1$ queue is not known to satisfy a Mehler formula of the type \eqref{eq:mehler-poisson}, which would make it rather difficult to differentiate directly. A Mehler-like formula can nevertheless be deduced from Theorem \ref{theorem_intertwining_order1}: choosing $u=1$ in this theorem, we derive 
$$
\mathbb{E}[ f(X^{x+1}_t) -f(X^{x}_t)]= \mathbb{E}\left[e^{-\int_0^t V(X^x_s) ds} (f(X^{x}_t+1) -f(X^{x}_t))\right],
$$
where $(X^x_t)_{t\geq 0}$ is M/M/1 process starting from $x$ and $V(x)=\beta \mathbf{1}_{x=0}$. As a consequence, if $B_t$ is a Bernoulli random variable verifying
\begin{align*}
\mathbb{P}(B_t=1 \ | \ (X^x_s)_{s\leq t}) &= e^{-\int_0^t V(X^x_s) ds},&t \geq 0,
\end{align*}
then,
\begin{align*}
\mathbb{E}[ f(X^{x+1}_t)]&= \mathbb{E}[(f(X^{x}_t+B_t)], &t \geq 0,
\end{align*}
and by induction there exists a random variable $Y^x_t$ such that
\begin{align*}
\mathbb{E}[ f(X^{x}_t)]&= \mathbb{E}[(f(X^{0}_t+ Y^x_t)], &t \geq 0.
\end{align*}

This formula seems to be new (even if the instantaneous distribution of the M/M/$1$ process is known, see \cite{baccelli1989sample}). Unfortunately, the random variable $Y^x_t$ is not independent from $X^{0}_t$ and this makes this formula less powerful than \eqref{eq:mehler-poisson}. This approach is generalizable for every BDP with constant birth rate (so that the processes $(X_{1,t})_{t\geq 0}$ and $(X_{t})_{t\geq 0}$ have the same law).\\

As in the preceding examples, we state a lemma related to the instantaneous probabilities of the modified process before turning to the Stein factors for geometric approximation.
\begin{lemma}[Upper bound of the instantaneous probabilities of a M/M/$1$ queue]
\label{lm:majo_mm1}
Let $(Y_t)_{t\geq 0}$ be a M/M/$1$ queue with rates $(\lambda,\lambda\1_{\N^*})$. Then for all $t\geq 0$,
$$\sup_{i \in \N^*}{\P(Y_t^i=i)} \leq \frac{1}{\sqrt{\lambda t}}.$$
\end{lemma}

\begin{proof}
Let us consider the BDP $(\t{Y}_t)_{t\geq 0}$ with rates $(1,\1_{x \in \N^*})$. Then for all $t\geq 0$, the equality in law $Y_t= \t{Y}_{\lambda t}$ holds, hence it is enough to prove that $\sup_{i \in \N^*}{\P( \t{Y}_t^i=i)} \leq \frac{1}{\sqrt{ t}}.$ By \cite[Corollary 1 (d)]{abate1991decompositions}, the sequence $(\P( \t{Y}_t^i=i))_{i\geq 0}$ is non-increasing for every $t\geq 0$. Hence $\sup_{i \in \N}{\P( \t{Y}_t^i=i)} = \P( \t{Y}_t^0=0)$. By \cite[formula $(9)$ and Corollary 2 (a)]{abate1991decompositions},
$$\P( \t{Y}_t^0=0)=\sum_{j=1}^\infty\frac{j}{t}\P(Z_t^0=j)=\frac{1}{t}\E[Z_t^0 \1_{Z_t^0 > 0}],$$
where $(Z_t^0)_{t\geq 0}$ is a birth-death process with constant birth rate $1$ and constant death rate $1$ on the whole integer line $\mathbb{Z}$; namely this is the continuous-time simple random-walk.
This process can be represented as
$$
\forall t\geq 0, \quad Z^0_t= N^+_t - N^-_t,
$$
where $(N^+_t)_{t\geq 0}$ and $(N^-_t)_{t\geq 0}$ are two independent Poisson processes with intensity $1$. So, using that $N^1$ and $N^2$ have the same law and Cauchy-Schwarz's inequality
\begin{align*}
\E[Z_t^0 \1_{Z_t^0 > 0}]
&=\E[(N^+_t - N^-_t)\1_{N^+_t > N^-_t}]= \E[(N^-_t - N^+_t)\1_{N^-_t > N^+_t}]\\
&=\frac{1}{2} \E[|N^+_t - N^-_t|]\leq \frac{1}{2} \text{Var}(N^+_t-N^-_t)^{1/2} =\sqrt{\frac{t}{2}}.
\end{align*}
This yields
$$\sup_{i \in \N^*}{\P( \t{Y}_t^i=i)} \leq \P( \t{Y}_t^0=0)\leq \frac{1}{\sqrt{2t}},$$
which achieves the proof.
\end{proof}

Up to the knowledge of the authors, Stein's factors associated to the Wasserstein distance have not been studied yet. The following proposition provides upper bounds on these factors.

\begin{proposition}[Estimation of the Stein's factors for Lipschitz function and geometric approximation] 
\label{thm:geometric_approximation_wasserstein}
For all $0<\a<\b$, set $u(x)=q^x$ on $\N$ with $q=\sqrt{\frac{\b}{\a}}=\rho^{-1/2}$. Then,

\begin{eqnarray*}
\sup_{f \in \text{Lip}(d_u)}\left\| \frac{g_f}{u}\right\|_\infty & =&\frac{1}{\sigma(u)}=\frac{1}{(\sqrt{\b}-\sqrt{a})^2},\\
\sup_{f \in \text{Lip}(d_u)}\|\p_v g_f\|_\infty &\leq & \frac{1}{(\sqrt{\b}-\sqrt{a})^2}\left(1+\sqrt{\frac{\a}{\b}}\min\left\{1,\frac{2\sqrt{\pi}}{(\a \b)^{1/4}}(\sqrt{\b}-\sqrt{a})-1\right\}\right).
\end{eqnarray*}

\end{proposition}

\begin{proof}[Proof of Proposition \ref{thm:geometric_approximation_wasserstein}]
By application of Theorem \ref{thm:stein_factor_wasserstein_order1}, one has immediately the first equation. By Theorem \ref{thm:stein_factor_wasserstein_order2} with $u(x)=q^x\, ,q=\rho^{-1/2}= \sqrt{\frac{\b}{\a}}$, we have:
\begin{equation*}
\sup_{f \in \text{Lip}(d_u)}\|\p_v g_f\|_\infty \leq  \int_0^\infty{e^{-(\sqrt{\b}-\sqrt{a})^2t}\left(1-\sqrt{\frac{\a}{\b}}+2\sqrt{\frac{\a}{\b}}\sup_{i \in \N^*}{\P(X_{1,*u,t}^i=i)}\right)dt},
\end{equation*}
where $(X_{1,*u,t}^i)_{t\geq 0}$ is a M/M/$1$ queue with rates $(\sqrt{\a \b},\sqrt{\a \b}\1_{\N^*})$. On the one hand, this yields directly
\begin{equation*}
\sup_{f \in \text{Lip}(d_u)}\|\p_v g_f\|_\infty \leq  \frac{1}{(\sqrt{\b}-\sqrt{a})^2}\left(1+\sqrt{\frac{\a}{\b}}\right).
\end{equation*}
On the other hand, as a consequence of Lemma \ref{lm:majo_mm1}, one has
\begin{eqnarray*}
\sup_{f \in \text{Lip}(d_u)}\|\p_v g_f\|_\infty & \leq & \int_0^\infty{e^{-(\sqrt{\b}-\sqrt{a})^2t}\left(1-\sqrt{\frac{\a}{\b}}+2\frac{1}{(\a \b)^{1/4}}\sqrt{\frac{\a}{\b}}\frac{1}{\sqrt{t}}\right)dt}\\
&=&\frac{1}{(\sqrt{\b}-\sqrt{a})^2}\left(1-\sqrt{\frac{\a}{\b}}\right)+\frac{1}{(\sqrt{\b}-\sqrt{\a})}\sqrt{\frac{\a}{\b}}\frac{2}{(\a \b)^{1/4}}\int_0^\infty{e^{-t}\frac{dt}{\sqrt{t}}}\\
&=&\frac{1}{(\sqrt{\b}-\sqrt{a})^2}\left(1-\sqrt{\frac{\a}{\b}}\right)+\frac{1}{(\sqrt{\b}-\sqrt{\a})}\sqrt{\frac{\a}{\b}}\frac{2\sqrt{\pi}}{(\a \b)^{1/4}}.
\end{eqnarray*}
\end{proof}

\begin{remark}[On the best upper bound]
The expression $\frac{2\sqrt{\pi}}{(\a \b)^{1/4}}(\sqrt{\b}-\sqrt{a})-1$ is smaller than $1$ as soon as 
$$\frac{\sqrt{\b}-\sqrt{a}}{(\a \b)^{1/4}} < \frac{1}{\sqrt{\pi}},$$
so there is a range of values of the parameters $\a$ and $\b$, for example if they are close to each other, for which the factor inside the $\min$ is actually a better upper bound than $1$.
\end{remark}

%The Stein's factors associated to indicator functions and bounded functions have been studied in \cite{pekoz2013total} in the case $\b(x)=1$ and $\a(x)=\rho$ :
%$$\sup_{f=\1_m, \,m\in \N}\|g_f\|_\infty \leq 1,\quad \sup_{0 \leq f \leq 1}\|\p g_f\|_\infty \leq 1.$$

We now turn to the subject of the mixture of geometric laws. Set $\varphi=1$ and $\rho<1$, then $\mathcal{I}_\varphi(\rho) =\mathcal{G}(\rho)$. We choose $u(k)=q^k$ on $\N$, hence $d_u(x,y)=|q^x-q^y|/|q-1|$. The preceding theorem put together with Theorem \ref{th:bornesI} gives for $q=\rho^{-1/2}$ and in the case where $\rho'<\sqrt{\rho},$
$$
W_{d_u} (\mathcal{G}(\rho), \mathcal{G}(\rho')) \leq |\rho - \rho'| \times \frac{1}{(1-\sqrt{\rho})^2} \times \frac{1-\rho'}{\sqrt{\rho}-\rho'}.
$$
The case $\rho'>\sqrt{\rho}$ is similar.\\

By the same reasoning as the one used in the proof of Theorem \ref{th:mixture}, for a random variable $R$ such that $\E[R]=\rho$, and a random variable such that $\mathcal{L}(W|R)=\mathcal{G}(R)$, we have the inequality:
\begin{align*}
d_\mathcal{F}(\mathcal{L}(W),\mathcal{G}(\rho))&\leq \sup_{f\in \mathcal{F}}\|\p_u g_f\|_\infty\E[(\rho-R)d_u(W+1,G+1)],
\end{align*}
where $G\sim\mathcal{G}(\rho)$. Let $G'\sim \mathcal{G}(\rho')$. With the interpretation of the geometric laws as the number of repetitions of a binary experiment before the first success, it is easy to find a coupling such that a.s. $G\leq G'$ when $\rho\leq \rho'$.
This yields
$$\E[d_u(G,G')]=\frac{1}{|1- q|}\left|\frac{1-\rho}{1-q\rho}-\frac{1-\rho'}{1-q\rho'}\right|=\frac{|\rho-\rho'|}{|(1-q\rho)(1-q\rho')|}.$$
Hence, if a.s. $R<\frac{1}{q}$, by Remark \ref{rq:mixture_unbiased_bis}:
\begin{align*}
d_\mathcal{F}(\mathcal{L}(W),\mathcal{G}(\rho))&\leq \sup_{f\in \mathcal{F}}\|\p_u g_f\|_\infty\frac{q}{1-q\rho}\E\left[\frac{|\rho-R|^2}{(1-qR)}\right].
\end{align*}

Finally, by taking $q=\rho^{-1/2}$, one finds that for two random variables $R,S$ such that $\E[R]=\rho$ and a.s. $R<\frac{1}{\sqrt{\rho}}$, and $\mathcal{L}(W|R)=\mathcal{G}(R)$, the following upper bound holds:
\begin{align*}
d_\mathcal{F}(\mathcal{L}(W),\mathcal{G}(\rho))\,\leq\, \frac{1+\frac{1}{\sqrt{\rho}}}{(1-\sqrt{\rho})^3}\,\E\left[\frac{|\rho-R|^2}{(1-\frac{R}{\sqrt{\rho}})}\right].
\end{align*}

%There exists a way of coupling two geometric laws of parameters $\rho,\rho'$ such that almost surely one is bigger than the other. This yields
%$$
%E[d_u(G,G')]=\frac{1}{|1- q|} \left| \frac{\rho q}{1- q+ q \rho} - \frac{\rho' q}{1- q+ q \rho'}  \right| \leq \frac{q}{1- q} |\rho-\rho'|.
%$$
%\Clairein{Je ne trouve pas la m\^eme chose. Je trouve $\E[d_u(G,G')]=\frac{1}{|1- q|}\left|\frac{1-\rho}{1-q\rho}-\frac{1-\rho'}{1-q\rho'}\right|=\frac{|\rho-\rho'|}{|(1-q\rho)(1-q\rho')|}$.}
%Thus by Kantorovich-Rubinstein theorem
%$$W_{d_u}(\mathcal{I}_\varphi(\Lambda) * \delta_1,\mathcal{I}_\varphi(\lambda)* \delta_1)\leq \E[d_u(G+1,G'+1)] = q \E[d_u(G,G')]\leq \frac{q^2}{1- q} |\rho-\rho'|.$$
%By plugging this in the preceding equation we find that
%$$
%|\E[S_\lambda g(W)]| \leq \Vert \partial_u g \Vert \frac{q^2}{1- q} \E\left[ |\lambda - \Lambda|^2 \right] = \Vert \partial_u g \Vert \frac{q^2}{1- q} \text{Var}(\Lambda).
%$$
%
%Notez que les mixtures (ou m\'elange) de loi g\'eom\'etrique apparaissent dans des r\'esultats de temps d'atteinte de processus de Markov, cf Diaconis, Miclo et Fill. Ici, on dit que si les valeurs propres sont assez proches le temps d'atteinte d'un point est presque g\'eom\'etrique.
%

\subsection{Another example}

\label{sect:sizeb}

Let us consider the BDP$(\a,\b)$ with $\alpha(x)=x+2, \beta(x)=x^2$ on $\N$. Its invariant measure is a Poisson size-biased type distribution, defined as
\begin{align*}
\pi(x)&=\frac{1}{2e}\frac{(x+1)}{x!},&x\in \N.
\end{align*}
Here size-biased means that if $X\sim \pi$ and $Y \sim \mathcal{P}(1)$ then:
\begin{align*}
\mathbb{P}(X=x)&= \frac{\mathbb{E}[(Y+1) \mathbf{1}_{Y=x}]}{\mathbb{E}[(Y+1)]} = \frac{(x+1) \mathbb{P}(Y=x)}{\sum_{j\geq 0} (j+1) \mathbb{P}(Y=j)}, &x\in \N.
\end{align*}
Choosing the weight $u$ such that $u(x+1)/u(x)=(x+1)/(x+3)$ for all $x\in \N$, i.e. for example
\begin{align*}
u(x) &= \frac{1}{(x+1)(x+2)}, &x\in \N,
\end{align*}
we find that $V_u$ is constant. By Theorem \ref{theorem_intertwining_order2_easy} with $v=1$, we have an intertwining with potential $V_{u,v} (x)=2x+1$ on $\N$. Moreover, by Theorem \ref{theorem_contraction_semigroup_zolotarev_distance}, we have convergence of the semigroup towards $\pi$ in the distance $\zeta_{u,1}$ at rate $1$. \\

The three next sections are devoted to the omitted proofs of the previous results.

\section{Proofs of Section \ref{sect:main}}

\label{sect:proof_main}
\subsection{First order intertwining for the backward gradient $\p_u^*$}
\label{sect:backward}

First of all, let us state the analogous of Theorem \ref{theorem_intertwining_order1} for the backward gradient $\p^*$. Let $(P_{*u,t})_{t\geq 0}$ be the birth-death semigroup associated to the generator $L_{*u}$, where for all non-negative or bounded function $f:\N \rightarrow \R$ and $x \in \N$, 
\begin{align*}
L_{*u} f&=\a_{*u}\, \p f+\b_{*u}\, \p^*f, &V_{*u}&=\bw{\a}-\a_{*u}+\b-\b_{*u},\\
\a_{*u}(x) &= \frac{u(x+1)}{u(x)}\a(x), & \b_{*u}(x) &=\frac{u(x-1)}{u(x)}\b(x-1)\1_{x\in \N^*}.\\
\end{align*} 
The potential $V_{*u}$ can be rewritten under the compacted form $V_{*u}=\p_u^*\left(\fw{u}\a-u \b\right)$. We can also notice that $V_{*u}=\Bw{V_{\fw{u}}}$ on $\N^*$. 

\begin{theorem}[First-order intertwining relation for the backward gradient]
\label{theorem_intertwining_order1_adjoint_gradient}
If $V_{*u}$ is bounded from below, then for every real-valued function on $\N$ such that $\|\p_u^*f\|_\infty <+\infty$, and for all $t\geq 0$,
\begin{equation}
\label{formula_intertwining_order1_adjoint_gradient_sg}
\p_u^* P_t f= P_{*u,t}^{V_{*u}} \,\p_u^* f.
\end{equation}
\end{theorem}

%When $u=1$, the semigroup in the left-hand side of \eqref{formula_intertwining_order1_adjoint_gradient_sg} has same jump rates as the semigroup in the right hand-side up to a slight shift in the death rate. In contrast, in Theorem \ref{theorem_intertwining_order1}, that was the birth rate that was shifted. It might be surprising and the following analytic proof, analogous to the proof of \cite{chafai2013intertwining}, does not really explain the reasons behind the particular expressions of the rates. The proof that we develop in Section \ref{sect:proofcoupling} gives a probabilistic interpretation of the intertwining relations and then leads to an interpretation of these shifts.

Let us call $({X}_{*u,t}^x)_{t\geq 0}$ the birth-death process of generator $L_{*u}$ such that ${X}_{*u,0}^x=x$. The process $({X}_{*u,t}^x)_{t\geq 0}$  is not irreducible, although it is indecomposable, i.e. it possesses only one recurrent class. Indeed if $x\in \N^*$ then $({X}_{*u,t}^x)_{t\geq 0}$ never visits the state $0$ as $\b_{*u}(1)=0$ and if $x=0$ the process $({X}_{*u,t}^0)_{t\geq 0}$ leaves $0$ almost surely. 
%Let $T^x$ be the first jump time of $({X}_{*u,t}^x)_{t\geq 0}$. For all $t \geq 0$, $\P({X}_{*u,t}^0>0)=\P(T^0\leq t)=1-e^{\a_{*u}(0)t}$. 

\begin{proof}[Proof of Theorem \ref{theorem_intertwining_order1_adjoint_gradient}]
The core of the proof relies on the intertwining relation at the level of generators:
\begin{equation}
\p_u^* Lf=L_{*u} \p_{u}^* f - V_{*u} \p_u^* f,
\label{formula_intertwining_order1_adjoint_gradient_gen}
\end{equation}
which derives by easy computations. The intertwining at the level of the semigroups follows by the same arguments as in the proof of Theorem $2.1$ of \cite{chafai2013intertwining}. We briefly recall these arguments. For all $s\in[0,t]$ let us set $J(s)=P_{*u,s}^{V_{*u}}(\p_u^* P_{t-s}f)$. If the function $\p_u^* P_{t-s}f$ is bounded on $\N$, then the Kolmogorov equations (\ref{equation_Kolomogorov_Feynamn-Kac}) for the Feynman-Kac semigroup $(P_{*u,t}^{V_{*u}})_{t\geq0}$ hold and 
$$J'(s)= P_{*u,s}^{V_{*u}}((L_{*u}-V_{*u})\p_{u}^*P_{t-s}f-\p_{u}^*LP_{t-s}f).$$ 
Thanks to the formula (\ref{formula_intertwining_order1_adjoint_gradient_gen}) this gives $J'(s)=0$. Hence $J(0)=J(t)$ which is exactly the identity (\ref{formula_intertwining_order1_adjoint_gradient_sg}).\\
Let us show that $\p_u^* P_{t-s}f$ is bounded on $\N$. Indeed, recall that $V_{*u}(x+1)={V_{\overset{\rightarrow}{u}}}(x)$ on $\N$. Furthermore $\p_u^*f(x+1)={\p_{\overset{\rightarrow}{u}}}f(x)$ on $\N$. Hence $V_{\overset{\rightarrow}{u}}$ and $\p_{\overset{\rightarrow}{u}}f$ are bounded on $\N$, which implies that $\p_{\overset{\rightarrow}{u}}P_{t-s}f$ is bounded (\cite{chen1996estimation}). For all positive integer $\p_{\overset{\rightarrow}{u}}P_{t-s}f(x)=\p_u^*P_{t-s}f(x+1)$, so $\p_u^*P_{t-s}f$ is bounded.
\end{proof}
%
%An analogous intertwining relation stands for the Neumann-type gradient $\partial^\star$. The alternative process that we find is actually the same, except that it is absorbed in $0$ instead of being reflected (i.e. $\a_{\s u}(0)=0$). It is non-irreducible and decomposable. The potential $V_{\s u}$ satisfies to $V_{\s u}=V_{*u}$ on $\N^*$.
 
\subsection{Alternative proof of first order intertwining theorems}
\label{sect:proofcoupling}

This section aims to give a sample path interpretation of the first order intertwining relations (\ref{formula_intertwining_sg_order1}) and (\ref{formula_intertwining_order1_adjoint_gradient_sg}), at least in a particular case. It is independent of the other sections.

We focus on the case where the weight is $u=1$ with non-increasing birth rates $(\a(x))_{x\in \N}$ and non-decreasing death rates $(\b(x))_{x\in \N}$.
When intertwining the birth-death semigroup with the forward gradient $\p$, one obtains a new birth-death semigroup with shifted birth rate and unchanged death rate
$$\a_{1}=\fw{\a},\quad \b_1=\b,$$ 
whereas when intertwining the birth-death semigroup with the backward gradient $\p^*$, one obtains a new birth-death semigroup with shifted death rate and unchanged birth rate:
$$\a_{*1}=\a,\quad \b_{*1}=\bw{\b}.$$
In order to explain this fact, we will give a probabilistic proof of the formulae (\ref{formula_intertwining_sg_order1}) and (\ref{formula_intertwining_order1_adjoint_gradient_sg}).
Recall that for all real-valued bounded functions on $\N$ and $x \in \N$,
\begin{eqnarray*}
\p P_t f(x) &=& \mathbb{E}[f(X_t^{x+1})-f(X_t^x)],\\
\p^* P_t f(x+1) &=& \mathbb{E}[f(X_t^{x})-f(X_t^{x+1})].
\end{eqnarray*} 
At time $t=0$, $X_t^{x+1}=X_t^x+1$. We construct a process $(S_t)_{t\geq 0}$ such that for all $t\geq 0$,  $X_t^{x+1}=X_t^x+S_t$ and $S_t\in \left\{0,1 \right\}$. If for a time $t$, $S_t=0$, then we choose the sticking coupling between $(X_{t+s}^x)_{s\geq 0}$ and $(X_{t+s}^{x+1})_{s\geq 0}$ (i.e. the process $(S_t)_{t\geq 0}$ is absorbed in $0$). If $S_t=1$, it is natural to construct the following coupling:
\begin{enumerate}
\item with rate $\a(X_t^{x}+1)=\a(X_t^{x+1})$, $X_t^x$ and $X_t^{x+1}$ jump upwards together and $S_t$ remains equal to $1$,
\item with rate $\b(X_t^{x})=\b(X_t^{x+1}-1)$, $X_t^x$ and $X_t^{x+1}$ jump downwards together and $S_t$ remains equal to \nolinebreak$1$,
\item with rate $\a(X_t^{x})-\a(X_t^{x}+1)=\a(X_t^{x+1}-1)-\a(X_t^{x+1})$, $X_t^x$ jumps upwards, $X_t^{x+1}$ does not jump and $S_t$ jumps from $1$ to $0$,
\item with rate $\b(X_t^{x+1})-\b(X_t^{x})=\b(X_t^{x+1})-\b(X_t^{x+1}-1)$, $X_t^{x+1}$ jumps downwards, $X_t^{x}$ does not jump and $S_t$ jumps from $1$ to $0$.
\end{enumerate}

This implies in particular that for all $t \geq 0$ the process $S_t$ jumps from $1$ to $0$ with rate
$$\a(X_t^{x})-\a(X_t^{x}+1) + \b(X_t^{x+1})-\b(X_t^{x})\,=\, V_1(X_t^x)\,=\,V_{*1}(X_t^{x+1}).$$ 
Moreover, conditionally to $\left\{S_t=1\right\}$, $(X^x_t)_{t\geq 0}$ evolves as a BDP$(\fw{\a},\beta)$ and  $(X^{x+1}_t)_{t\geq 0}$ evolves as a BDP$(\alpha,\bw{\b}))$. Indeed, as long as $S_t=1$, the steps $(3)$ and $(4)$ do not occur.\\

To exploit rigorously the preceding facts, let us introduce the BDP$(\fw{\a},\beta)$ starting from $x$ denoted by $(X_{1,t}^x)_{t\geq 0}$, whose standard filtration is $(\mathcal{F}_t)_{t\geq 0}$. The processes $(X_{t}^x)_{t\geq 0}$ and $(X_{1,t}^x)_{t\geq 0}$, as well as $(X_{t}^{x+1})_{t\geq 0}$ and $(X_{1,t}^x+1)_{t\geq 0}$, can be coupled as follows : 
\begin{enumerate}
\item Let $E$ be an exponential with parameter $1$ and $T$ such that
$
T= \inf\{t \geq 0 , \int_0^t V(X_{1,s}^x) ds > E \}.
$
\item Set $S_t=1$ if $t< T$ and $S_t=0$ otherwise.
\item Set $X^x_t= X_{1,t}^x$ for $t\leq T$.
\item At time $T$, sample a random variable $Z$ satisfying to
\begin{align*}
\mathbb{P}(Z= X_{1,T}^x +1 \ | \ \mathcal{F}_T ) &= \frac{\alpha(X_{1,T}^x) -\alpha(X_{1,T}^x+1)}{V(X_{1,T}^x)},\\
 \mathbb{P}(Z= X_{1,T}^x \ | \ \mathcal{F}_T ) &= \frac{\beta(X_{1,T}^x+1) -\beta(X_{1,T}^x)}{V_1(X_{1,T}^x)}.
\end{align*}

\item Let evolve the process $(X^x_t)_{t\geq T}$ as a BDP$(\alpha,\beta)$ starting from $Z$.
\end{enumerate}
The coupling $(X_t^x,X_{1,t}^x,S_t)_{t\geq 0}$ satisfy to
\begin{align*}
X_t^{x}\mathbf{1}_{S_t=1}&= X_{1,t}^x\mathbf{1}_{S_t=1},& X_t^{x+1}\mathbf{1}_{S_t=1}&=(X_{1,t}^x+1)\mathbf{1}_{S_t=1},& \mathbb{P}(S_t=1|(X_{1,s}^x)_{0\leq s \leq t})&=e^{-\int_0^t{V_1(X_{1,s}^x)ds}}.
\end{align*}

This allows to find back the formula (\ref{formula_intertwining_sg_order1}):
\begin{eqnarray*}
\p P_t f(x)&=& \mathbb{E}[f(X_t^{x+1})-f(X_t^x)] =\E\left[(f(X_t^{x+1})-f(X_t^x))\1_{S_t=1}\right]\\
&=&\E\left[(f(X_{1,t}^x+1)-f(X_{1,t}^x))e^{-\int_0^t{V_1(X_{1,s}^x)ds}}\right]=P_{1,t}^{V_1}(\p f)(x).
\end{eqnarray*}

Similarly it is possible to construct a coupling $(X_t^{x+1},X_{*1,t}^{x+1},S_t)_{t\geq 0}$ such that $(X_{*1,t}^{x+1})_{t\geq 0}$ is a BDP$(\alpha,\bw{\b})$ starting from $x+1$ and satisfying to
\begin{align*}
X_t^{x+1}\mathbf{1}_{S_t=1}&=X_{*1,t}^{x+1}\mathbf{1}_{S_t=1},& X_t^{x}\mathbf{1}_{S_t=1}&=(X_{*1,t}^{x+1}-1)\mathbf{1}_{S_t=1},&\mathbb{P}(S_t=1|(X_{*1,s}^{x+1})_{0\leq s \leq t})=e^{-\int_0^t{V_{*1}(X_{*1,s}^{x+1})ds}},
\end{align*}
leading to the formula (\ref{formula_intertwining_order1_adjoint_gradient_sg}):
\begin{eqnarray*}
\p^* P_t f(x+1)&= &\mathbb{E}[f(X_t^{x})-f(X_t^{x+1})]= \mathbb{E}[(f(X_t^{x})-f(X_t^{x+1}))\1_{S_t=1}] \\
&=&\E\left[(f(X_{*1,t}^{x+1}-1)-f(X_{*1,t}^{x+1}))e^{-\int_0^t{V_{*1}(X_{*1,s}^{x+1})ds}}\right]=P_{*1,t}^{V_{*1}}(\p^*f)(x+1).\\
\end{eqnarray*}

It is interesting to remark that conversely, the intertwining formula (\ref{formula_intertwining_sg_order1}) can in certain cases yield a coupling between $(X_t^x)_{t\geq0}$ and $(X_t^{x+1})_{t\geq0}$. The proof of Lemma \ref{lemma_mehler_negbin} above is based on this idea.

\subsection{Proof of Theorems \ref{theorem_intertwining_order2_hard} and \ref{theorem_intertwining_order2_easy}}
\begin{proof}[Proof of Theorem \ref{theorem_intertwining_order2_hard}]
Let us begin by showing the following intertwining relation at the level of the generators:
\begin{equation*}
\p_{v}^*\p_u Lf=L_{u,*v} \p_{v}^*\p_u f - V_{u,*v} \p_{v}^*\p_u f.
\end{equation*}
By application of Theorem \ref{theorem_intertwining_order1} and Theorem \ref{theorem_intertwining_order1_adjoint_gradient} we find that
\begin{eqnarray*}
\p_v^*(\p_uLf)&=&\p_v^*(L_u(\p_u f)-V_u \p_uf )\\
&=&(L_u)_{*v}\p_v^*\p_u f-(V_u)_{*v}\p_v^*\p_u f + \p_v^*(-V_u \p_u f),
\end{eqnarray*}
where $(L_u)_{*v}$ and $(V_u)_{*v}$ stand for the generator, respectively the potential, obtained by intertwining the BDP$(\a_u,\b_u)$ and the $\p_{*v}$ gradient.
The generator $(L_{u})_{*v}$ is the generator of a BDP$((\a_u)_{*v},(\b_u)_{*v})$ such that for all $x\in \N$,
\begin{eqnarray*}
(\a_u)_{*v}(x)&=&\frac{v(x+1)}{v(x)}\a_u(x)=\frac{v(x+1)}{v(x)}\frac{u(x+1)}{u(x)}\a(x+1)\\
(\b_u)_{*v}(x)&=&\frac{v(x-1)}{v(x)}\b_u(x-1)=\frac{v(x-1)}{v(x)}\frac{u(x-2)}{u(x-1)}\b(x-1)\1_{x\in \N^*}.
\end{eqnarray*}
The potential $(V_u)_{*v}$ writes on $\N$
\begin{eqnarray*}
(V_u)_{*v}(x)&=&\a_u(x-1)\1_{x\in \N^*}-(\a_u)_{*v}(x)+\b_u(x)-(\b_u)_{*v}(x).
\end{eqnarray*}
The next step is to rewrite the expression $\p_v^*(-V_u \p_u f)$ in terms of $\p_v^*\p_u f$. Let us denote $g=\p_u f$ in the following lines. For every $x \in \N^*$, $\p_u^*(fg)(x)=f(x)\p_u^* g(x)+\p_u^* f(x)g(x-1)$ and $f(x)=-\sum_{k=0}^x{u(k)\p_u^*f(k)}$ so that
\begin{eqnarray*}
\p_v^*(-V_u g)(x)&=&-V_u(x)\p_v^*g(x)-\p_v^*V_u(x)g(x-1)\\
&=&-V_u(x)\p_v^*g(x)+\p_v^*V_u(x)\sum_{k=0}^{x-1}{v(k)\p_v^* g (k)}\\
&=&\p_v^*V_u(x)\sum_{k=0}^{x-1}{v(k)(\p_v^* g (k)-\p_v^* g (x))}-\left(V_u(x)-\left(\sum_{k=0}^{x-1}{v(k)}\right)\p_v^*V_u(x)\right)\p_v^*g(x).
\end{eqnarray*}
Besides, $\p_v^*(-V_u g)(0)=\frac{1}{v_0}V_u(0)g(0)=-V_u(0)\p_v^*g(0)$. We do indeed find $\p_v^*\p_u L=(L_{u,*v}-V_{u,*v})\p_v^*\p_u$ with
\begin{eqnarray*}
L_{u,*v}f(x)&=&(L_u)_{*v}f(x)+\p_v^* V_u(x)v(x-1)(f(x-1)-f(x))\\
&&+\p_v^* V_u(x)\left(\sum_{j=0}^{x-2}{v(j)}\right)\sum_{k=0}^{x-2}{\frac{v(k)}{\left(\sum_{j=0}^{x-2}{v(j)}\right)}(f(k)-f(x))}\\
V_{u,*v}(x)&=&(V_u)_{*v}(x)+V_u(x)-\left(\sum_{k=0}^{x-1}{v(k)}\right)\p_v^*V_u(x).
\end{eqnarray*}
The generator $L_{u,*v}$ has a birth-death component and a component making the process at point $x$ jumping on the set $\left\{0,\dots ,x-2 \right\}$. The birth rates are $\a_{u,*v}=(\a_u)_{*v}$. The death rates come from $(L_u)_{*v}$ and from the term $\p_v^* V_u(x)v(x-1)(f(x-1)-f(x))$, so that
$$\b_{u,*v}(x)=(\b_u)_{*v}+\p_v^* V_u(x)v(x-1)\1_{x \in \N^*}.$$
Remembering that $V_u(x)=\a(x)-\a_u(x)+\b(x+1)-\b_u(x)$ we get that for all positive integer $x$
\begin{eqnarray*}
(V_u)_{*v}(x)+V_u(x)&=&\a(x)+\a_u(x-1)-(\a_u(x)+(\a_u)_{*v}(x))+\b(x+1)-(\b_u)_{*v}(x)\\
&=&\left(1+\frac{u(x)}{u(x-1)}\right)\a(x)-\left(1+\frac{v(x+1)}{v(x)}\right)\frac{u(x+1)}{u(x)}\a(x+1)\\
&&+\b(x+1)-\frac{v(x-1)}{v(x)}\frac{u(x-2)}{u(x-1)}\b(x-1),
\end{eqnarray*}
and
\begin{eqnarray*}
(V_u)_{*v}(0)+V_u(0)&=&-(\a_u)_{*v}(0)+V_u(0)=\a(0)-(\a_u(0)+(\a_u)_{*v}(0))+\b(1)\\
&=&\a(0)-\left(1+\frac{v(1)}{v(0)}\right)\frac{u(1)}{u(0)}\a(1)+\b(1).
\end{eqnarray*}
The same reasoning as in the proof of Theorem \ref{theorem_intertwining_order1_adjoint_gradient} allows to deduce the relation at the level of the semigroups from the relation at the level of the generators, provided that we can show that for all $t\geq 0$ the function $\p_v^*\p_u P_t f$ is bounded on $\N$. It is the case; indeed, by Theorem \ref{theorem_intertwining_order1}, $\p_u P_t f=P_{u,t}^{V_u}\,{\p_u f}$ is bounded and $\p_v^*|\p_u P_t f|\leq \frac{2}{\inf_{x\in \N}{v(x)}}|P_{u,t}^{V_u}\,{\p_u f}|$.
\end{proof}

\begin{proof}[Proof of Theorem \ref{theorem_intertwining_order2_easy}]
Surprisingly, Theorem \ref{theorem_intertwining_order2_easy} cannot be deduced from Theorem  \ref{theorem_intertwining_order2_hard} when $u \neq 1$. However, its proof goes along the same lines as the proof of Theorem \ref{theorem_intertwining_order2_hard}, only easier because $\p_v\p_u (V_u \p_u f)=V_u \p_v\p_u f$, so that the intertwining relation at the level of the generators follows directly.
\end{proof}

\section{Proofs of Section \ref{sect:stein}}
\label{sect:proof_stein}

The semigroup representation \eqref{eq:stein_sg_representation} of the solution of Stein's equation $g_f$ can be rewritten as:
\begin{eqnarray}
\fw{g_f}&=& - u \int_0^\infty \p_u P_tf dt, \label{eq:sol_stein_0}\\
\p g_f&=& u \int_0^\infty \p_u \p^* P_tf dt,\label{eq:sol_stein_1}\\
\Fw{\p g_f}&=& - u \int_0^\infty \p_u \p P_tf dt.\label{eq:sol_stein_1bis}
\end{eqnarray}
%\begin{eqnarray}
%\frac{1}{u}\fw{g_f}&=& -\int_0^\infty \p_u P_tf dt \label{eq:sol_stein_0}\\
%\p_u g_f&=&\int_0^\infty \p_u \p^* P_tf dt\label{eq:sol_stein_1}\\
%\frac{1}{u}\Fw{\p g_f}&=& -\int_0^\infty \p_u \p P_tf dt.\label{eq:sol_stein_1bis}
%\end{eqnarray}

The left-hand side of an intertwining relation between a weighted gradient and a birth-death semigroup appears under the integral. This fact suggests to apply the intertwining relations shown previously. However, it leads to sharper results to first identify the function $f\in \mathcal{F}$ that realizes the maximum in the pointwise Stein's factors
$$
\max_{f \in \mathcal{F}} | g_f(i)|, \quad \max_{f \in \mathcal{F}} | \partial g_f(i)|,
$$ 
for every $i\in \N$. This first step is based on Lemma \ref{lemma:bx1} and Lemma \ref{lemma:bx2} below. Indeed, Lemma \ref{lemma:bx1} gives an alternative formulation of the solution of Stein's equation.

%Let us recall the results of \cite{brown2001stein} that we will use to estimate the pointwise Stein's factor.

\begin{lemma}[{\cite[Lemma 2.3]{brown2001stein}}]
\label{lemma:bx1}
For all $i \in \N$, let us define $g_j:=g_{\1_j}$ and
\begin{align*}
e_{i}^+&=\frac{1}{\a(i) \pi(i)}\sum_{k=0}^i\pi(k),\quad i \in \N & e_i^-&=\frac{1}{\b(i) \pi(i)}\sum_{k=i}^\infty\pi(k),\quad i \in \N^*.
\end{align*}
Then, for all $i\in \N^*,j\in \N$,
\begin{align}
g_j(i)&=\pi(j) (-e_{i-1}^+\1_{i\leq j}+e_i^-\1_{i \geq j+1})\label{eq:sol_stein_0_bx}\\
\p g_j(i) &=\pi(j)\left((e_{i-1}^+-e_i^+)\1_{j \geq i+1}+(e_{i+1}^-+e_{i-1}^+)\1_{i=j}+(e_{i+1}^--e_i^-)\1_{j \leq i-1}\right).\label{eq:sol_stein_1_bx}
\end{align}
\end{lemma}

\begin{lemma}[{\cite[Lemma 2.4]{brown2001stein}}]
\label{lemma:bx2}
If $V_1\geq 0$ then $(e_i^+)$ is non-decreasing and $(e_i^-)$ is non-increasing.
\end{lemma}

%In a second step, we use equations \eqref{eq:sol_stein_0}-\eqref{eq:sol_stein_1bis} and the intertwining theorems to evaluate the maximum of the pointwise Stein's factors. Theorem \ref{theorem_intertwining_order1} gives results for the first Stein's factor and our main result Theorem \ref{theorem_intertwining_order2_hard}, with its variant Theorem \ref{theorem_intertwining_order2_easy}, gives results for the second Stein's factor.
%
%To do these two steps, we now deal successively with the estimation of the Stein's factors for three classes of functions $\mathcal{F}$: 
%\begin{align*}
%\mathcal{F}&=\left\{0\leq f\leq 1\right\},&\mathcal{F}&=\text{Lip}(d_u),&\mathcal{F}&=\left\{f=\1_{[0,m]},\,m\in\N\right\}.
%\end{align*}
%
%
%These classes correspond, respectively, to the total variation, Wasserstein and Kolmogorov distances.
%
%\iffalse
%A fact that the reader should keep in mind to compare the results is that as for the distances \eqref{eq:majo_stein_wasserstein}, the Stein's factors associated to the different classes compare between each other:
%\begin{remark}[Comparison between the Stein's factors for different classes $\mathcal{F}$]
%\begin{align*}
%\sup_{0\leq f\leq 1}\|g_f\|_\infty &\leq \frac{1}{\inf_\N u}\sup_{f\in \text{Lip}(d_u)}\|g_f\|_\infty,	 & \sup_{f=\1_{[0,m]},\,m\in\N}\|g_f\|_\infty &\leq \sup_{0\leq f\leq 1}\|g_f\|_\infty.
%\end{align*}
%The analogous comparisons hold for the second Stein's factors.
%\end{remark}
%\fi

\subsection{Approximation in total variation distance}

We begin by describing the argmax of the pointwise quantities. To the knowledge of the authors, equation \eqref{eq:argmax_bdd} is not explicitly stated in preceding works. Equation \eqref{eq:argmax_derivative_bdd} is proved in \cite{brown2001stein}. We briefly recall the arguments used for the sake of completeness.

\begin{lemma}[Argmax of the pointwise Stein's factor]
\label{lemma:argmaxtv}
For all $i \in \N$,
\begin{eqnarray}
g_{\1_{[0,i]}}(i) =\sup_{0\leq f\leq 1}{\Fw{g_f}(i)}.
\label{eq:argmax_bdd}
\end{eqnarray}

Moreover if $V_1\geq 0$, then for all $i\in \N^*$
\begin{eqnarray}
\p g_{\1_i} (i) &= &\underset{0\leq f \leq 1}{\max}\, |\p g_f (i)|.\label{eq:argmax_derivative_bdd}%\\
\end{eqnarray}
\end{lemma}

\begin{proof}
By replacing $f$ by $1-f$ if necessary
$$\underset{0\leq f \leq 1}{\sup}\, | g_f (i)|=\underset{0\leq f \leq 1}{\sup}\,  g_f (i),\quad \underset{0\leq f \leq 1}{\sup}\, |\p g_f (i)|=\underset{0\leq f \leq 1}{\sup}\, \p g_f (i).$$
By Lemma \ref{lemma:bx1},
\begin{align*}
g_f(i+1)&= e_{i+1}^-\sum_{j=0}^{i}{\pi(j)f(j)}-e_{i}^+\sum_{j=i+1}^\infty{\pi(j)f(j)} \leq e_{i+1}^-\sum_{j=0}^{i}{f(j)},
\end{align*}
with equality for $f=\1_{[0,i]}$ which proves \eqref{eq:argmax_bdd}. On the other hand,
$$\p g_j(i)=\pi_j\left((e_{i-1}^+-e_i^+)\1_{i \leq j-1}+(e_{i+1}^-+e_{i-1}^+)\1_{i=j}+(e_{i+1}^--e_i^-)\1_{i \geq j+1}\right),$$
so by Lemma \ref{lemma:bx2} the quantity $\p g_j(i)$ is non-negative if and only if $i=j$. Hence, if $f$ is a function on $\N$ with values in $[0,1]$,
$$\p g_f(i)=\sum_{j=0}^\infty{f(j)\p g_j(i)}\leq \p g_i(i),$$
and there is equality if $f=\1_i$. This shows \eqref{eq:argmax_derivative_bdd}.
\end{proof}

As a consequence, we have the following lemma of which Theorem \ref{thm:stein_factor_tv_order1} is a direct application.

\begin{lemma}[Pointwise first Stein's factor for bounded functions]
\label{thm:stein_factor_tv_0}
If $V_u$ is bounded from below by $\sigma(u)$ then for all $i\in \N$,
\begin{eqnarray*}
\sup_{0\leq f\leq 1}{|g_f(i+1)|}\leq \int_0^\infty{e^{-\sigma(u)t}\P(X_{u,t}^i=i)dt}.
\end{eqnarray*}
Moreover if $V_u$ is constant then the preceding inequality is in fact an equality.
\end{lemma}

\begin{proof}%[Proof of Theorem \ref{thm:stein_factor_tv_0}]
By the equation \eqref{eq:sol_stein_0}, Theorem \ref{theorem_intertwining_order1} and Lemma \ref{lemma:argmaxtv}, and because $\p_u \1_{[0,i]}=-\frac{1}{u(i)}\1_i$, we have for all function $f$ such that $0\leq f \leq 1$
\begin{align*}
|g_f(i+1)|&\leq g_{\1_{[0,i]}}(i+1)&=-u(i)\int_0^\infty{P_{u,t}^{V_u}(\p_u \1_{[0,i]})}=\int_0^\infty{P_{u,t}^{V_u}(\1_i)}\leq \int_0^\infty{e^{-\sigma(u)t}\P(X_{u,t}^i=i)dt}.
\end{align*}
\end{proof}

We now state results for the second pointwise Stein factor. 

\begin{lemma}[Pointwise second Stein's factor for bounded functions]
\label{lem:stein_factor_tv}

\begin{itemize}
\item[]

\item Under $\mathbf{H_1}$, for all integer $i\in \N^*$, the quantity $\sup_{0\leq f \leq 1}|\p g_f(i)|$ is bounded by
\begin{eqnarray*}
\int_0^\infty{e^{-\sigma(1,*u)t}\left(-\frac{u(i)}{u(i-1)}\P(X_{1,*u,t}^i=i-1)+2\P(X_{1,*u,t}^i=i)-\frac{u(i)}{u(i+1)}\P(X_{1,*u,t}^i=i+1)\right)dt}.
\end{eqnarray*}

\item Under $\mathbf{H_2}$, for all integer $i\in \N$, the quantity $\sup_{0\leq f \leq 1}|\p g_f(i+1)|$ is bounded by
\begin{eqnarray*}
\int_0^\infty e^{-\sigma(1,u)t} \left( -\frac{u(i)}{u(i-1)}\P(X_{1,u,t}^i=i-1) + 2\P(X_{1,u,t}^i=i)-\frac{u(i)}{u(i+1)}\P(X_{1,u,t}^i=i+1)\right)dt.
\end{eqnarray*}

\end{itemize}
Moreover, if the potential $V_{1,*u}$ (respectively $V_{1,u}$) is constant, then the first (respectively the second) upper bound is in fact an equality.
\end{lemma}

\begin{proof}% [Proof of Theorem \ref{theorem_stein_factor_tv}]
For every positive integer $i$, let $f_i=\1_i$. By the equation \eqref{eq:sol_stein_1}, Theorem \ref{theorem_intertwining_order2_hard} and Lemma \ref{lemma:argmaxtv}, under $\mathbf{H_1}$,
$$\sup_{0\leq f\leq 1}|\p g_{f}(i)|=\p g_{f_i}(i)=u(i)\p_u g_{f_i}(i)\leq u(i)\int_0^\infty{e^{-\sigma(1,*u)t} \mathbb{E}\left[\p_u^* \p f_i(\tilde{X}_t^i)\right] dt}.$$
As $\p_u^* \p f_i=-\frac{1}{u(i-1)}\1_{i-1}+2\frac{1}{u(i)}\1_{i}-\frac{1}{u(i+1)}\1_{i+1}$, we get the announced inequality.

Similarly the result under $\mathbf{H_2}$ derives from the equation \eqref{eq:sol_stein_1bis}, Theorem \ref{theorem_intertwining_order2_easy}, Lemma \ref{lemma:argmaxtv} and the computation $-\p_u \p f_{i+1}=-\frac{1}{u(i-1)}\1_{i-1}+2\frac{1}{u(i)}\1_{i}-\frac{1}{u(i+1)}\1_{i+1}$. 
%The reasoning under $\mathbf{H_2}$ goes along the same lines: for all $i \in \N$,
%$$\sup_{0\leq f\leq 1}|\p g_{f}(i+1)|\p g_{f_{i+1}}(i+1)=-u(i)\int_0^\infty{\p_u \p P_t f_{i+1}dt}\leq u(i)\int_0^\infty{e^{-\sigma(1,u)t}P_{1,u,t}(-\p_u \p f_{i+1})dt}$$
%and $-\p_u \p f_{i+1}=-\frac{1}{u(i-1)}\1_{i-1}+2\frac{1}{u(i)}\1_{i}-\frac{1}{u(i+1)}\1_{i+1}$.
\end{proof}

Theorem \ref{thm:stein_factor_tv_order2} and \ref{thm:stein_factor_tv_order2_bis} are direct consequences of Lemma \ref{lem:stein_factor_tv}.

\subsection{Approximation in Wasserstein distance.}

In contrast with the first order in total variation distance, the bound of Theorem \ref{thm:stein_factor_wasserstein_order1} does not require a preliminary bound on pointwise Stein's factor. 

\begin{proof}[Proof of Theorem \ref{thm:stein_factor_wasserstein_order1}] 
By Theorem \ref{theorem_intertwining_order1},
\begin{eqnarray*}
 \left|\frac{1}{u}g_f(\cdot+1)\right|&=& |\p_u h_{f}| =\left| \int_0^\infty \p_u  P_t f dt \right|\\
 &=&\left| \int_0^\infty   \E\left[e^{-\int_0^t V_u(X_u,s)ds}\p_u f(X_{u,t})\right] dt \right| \leq \frac{1}{\sigma(u)} \Vert \partial_u f\Vert_\infty.
\end{eqnarray*}
Now to prove the sharpness if $V_u$ is constant, it is enough to consider the map $f:x\mapsto \nolinebreak -\sum_{k=1}^x u(x- \nolinebreak 1)$ for which the previous inequalities are in fact equalities.
\end{proof}

\begin{remark}[Variant of Theorem \ref{thm:stein_factor_wasserstein_order1}]
We can also derive an upper bound for 
$$\sup_{f\in \lip(d_u)} \Vert {g_f} /u \Vert_\infty,$$ under the condition that $V_{*u}$ is bounded by below, by using alternatively to equation \eqref{eq:sol_stein_0} the equation
\begin{align*}
{g_f}&= - u \int_0^\infty \p_u^* P_tf dt,
\end{align*}
and Theorem \ref{theorem_intertwining_order1_adjoint_gradient} instead of Theorem \ref{theorem_intertwining_order1}.
\end{remark}

%In contrast with the first order in Wasserstein distance,
%the proof of Theorem \ref{thm:stein_factor_wasserstein_order2} and Theorem \ref{thm:stein_factor_wasserstein_order2_bis} needs the two-step logic exposed in Section \ref{sect:stein-gen}.

For the second Stein factor, we begin by focusing on the pointwise quantity $\sup_{f\in \mathcal{F}}\p_u g_f(i)$. For all $i\in \N$, let us introduce two functions $\psi_i$ and $\Psi_i$ defined for all $j \in \N$ as
\begin{eqnarray*}
\psi_i(j)&=&\left(1-\frac{u(j-1)}{u(j)}\right)\1_{ j \leq i-1 }+\left(1+\frac{u(j-1)}{u(j)}\right)\1_{j=i}+\left(\frac{u(j-1)}{u(j)}-1\right)\1_{j \geq i+1}\\
\Psi_i(j)&=&\left(1-\frac{u(j+1)}{u(j)}\right)\1_{ j \leq i-1 }+\left(1+\frac{u(j+1)}{u(j)}\right)\1_{j=i}+\left(\frac{u(j+1)}{u(j)}-1\right)\1_{j \geq i+1}.
\end{eqnarray*}

The following lemma allows to determine the functions that realize the supremum in the second pointwise Stein factor. This lemma is a generalization of a lemma of \cite{barbour2006stein}, which addressed the case where $u=1$ and $(\a(x),\b(x))_{x \in \mathbb{N}}=(\lambda,x)_{x \in \mathbb{N}}$. Its proof depends on the already cited results of \cite{brown2001stein}.

\begin{lemma}[Argmax of the pointwise Stein's factor]
\label{lemma:argmax}
If $V_1\geq 0$, then for all $i\in \N^*$
\begin{align}
\p g_{\varphi_i} &= \underset{f\in \lip(d_u)}{\max}\, |\p g_f (i)|,& \varphi_i&=-d_u(i,\cdot).
\label{eq:argmax_derivative_lip}
\end{align}
\end{lemma}

\begin{proof}
If $f$ and $\t{f}$ are two real-valued functions on $\N$, then $g_{f+\t{f}}=g_f+g_{\t{f}}$ and that if $f$ is constant, then $g_f=0$. As a consequence, by replacing $f$ by $-f$ and $f-f(i)$ if necessary,
$$\underset{f\in \lip(d_u)}{\sup}|\p g_f (i)|=\underset{\substack{f\in \lip(d_u),\\f(i)=0}}{\sup}\, \p g_f (i).$$
Recall that $g_j:=g_{\1_j}$ for $j\in \N$. For all real-valued function  $f$ on $\N$,
$$g_f=\sum_{j\in \N}f(j)g_j,\quad \quad \p g_f(i)=\sum_{j\in \N}f(j)\p g_j(i),\quad i \in \N.$$
By Lemmas \ref{lemma:bx1} and \ref{lemma:bx2}, if $f\in \lip(d_u)$ and $f(i)=0$ then
\begin{eqnarray*}
\p g_f(i)&=&\sum_{j=0}^\infty{f(j)\p g_j(i)}=(e_{i-1}^+-e_{i}^+)\sum_{j \leq i-1}{\pi_j f(j)}+(e_{i+1}^--e_{i}^-)\sum_{j \geq i+1}{\pi_j f(j)}\\
&\leq &|\p g_{\varphi_i}(i)| =(e_{i}^+-e_{i-1}^+)\sum_{j \leq i-1}{\pi_j d_u(i,j)}+(e_{i}^--e_{i+1}^-)\sum_{j \geq i+1}{\pi_j d_u(i,j)}.
\end{eqnarray*}
%with equality for $f=\varphi_i$.
\end{proof}

\begin{lemma}[Pointwise second Stein's factor for Lipschitz functions]
\label{theorem_stein_factor_wasserstein}
\begin{itemize}
\item[]

\item Under $\mathbf{H_1}$, for all integer $i\in \N^*$,
\begin{equation}
\sup_{f \in \text{Lip}(d_u)}|\p_u g_f(i)|\leq \int_0^\infty e^{-\sigma(1,*u)t}\mathbb{E}[\psi_i(X_{1,*u,t}^i)]dt.
\label{equation_majoration_stein_factor}
\end{equation}
Moreover if $V_{1,*u}$ is constant then the preceding inequality is in fact an equality.

\item Under $\mathbf{H_2}$, for all integer $i\in \N$
\begin{equation}
\sup_{f \in \text{Lip}(d_u)}\left|\frac{1}{u(i)}\p g_f(i+1)\right|\leq \int_0^\infty e^{-\sigma(1,u)t} \mathbb{E}[\Psi_i(X_{1,u,t}^i)]dt.
\label{equation_majoration_stein_factor_bis}
\end{equation} 
Moreover if $V_{1,u}$ is constant then the preceding inequality is in fact an equality.
\end{itemize}
\end{lemma}

\begin{proof}[Proof of Lemma \ref{theorem_stein_factor_wasserstein}]
Let us assume that $\mathbf{H_1}$ holds true. By the equation \eqref{eq:sol_stein_1}, Theorem \ref{theorem_intertwining_order2_hard} and Lemma \ref{lemma:argmax}, for every positive integer $i$,
\begin{eqnarray*}
\sup_{f\in \lip(d_u)}|\p_u g_f(i)|&=&\frac{1}{u(i)}\sup_{f\in \lip(d_u)}|\p g_f(i)|=\frac{1}{u(i)}\p g_{\varphi_i}(i)=\p_u g_{\varphi_i}(i)=\int_0^\infty{\p_u^* \p P_t \varphi_i dt}\\
&=&\int_0^\infty{ P_{1,*u,t}^{V_{1,*u}}(\p_u^* \p \varphi_i) dt} \leq \int_0^\infty{e^{-\sigma(1,*u)t} \mathbb{E}\left[\p_u^* \p \varphi_i(X_{1,*u,t}^i)\right] dt}.
\end{eqnarray*}
It is easy to check that $\psi_i=\p_u^* \p \varphi_i$, which proves \eqref{equation_majoration_stein_factor}. 

Now, if $\mathbf{H_2}$ holds true, by the equation \eqref{eq:sol_stein_1bis}, Theorem \ref{theorem_intertwining_order2_easy} and Lemma \ref{lemma:argmax}, for all integer $i$,
\begin{eqnarray*}
\sup_{f\in \lip(d_u)}\left|\frac{1}{u(i)}\p g_f(i+1)\right|&=&\frac{1}{u(i)}\sup_{f\in \lip(d_u)}|\p g_f(i+1)|=\frac{1}{u(i)}\p g_{\varphi_{i+1}}(i+1)\\
&=&-\int_0^\infty{\p_u \p P_t \varphi_{i+1}(i) dt}\\
&=&-\int_0^\infty{ P_{1,u,t}^{V_{1,u}} \p_u \p \varphi_{i+1}(i) dt}
\end{eqnarray*}
As $-\p_u \p \varphi_{i+1}=\Psi_i$, the equation \eqref{equation_majoration_stein_factor_bis} holds true.
\end{proof}

We deduce from Lemma \ref{theorem_stein_factor_wasserstein} both Theorem \ref{thm:stein_factor_wasserstein_order2} and Theorem \ref{thm:stein_factor_wasserstein_order2_bis}. We only give the proof of Theorem \ref{thm:stein_factor_wasserstein_order2} because Theorem \ref{thm:stein_factor_wasserstein_order2_bis} is similar.

\begin{proof}[Proof of Theorem \ref{thm:stein_factor_wasserstein_order2}]
First of all let us notice that for all function $f:\N\rightarrow\R$,
\begin{align}
\|\p_uf\|_\infty &\leq \sup_{x\in \N}\left(1+\frac{u(x+1)}{u(x)}\right)\|f/u\|_\infty, &\|\p_u^*f\|_{\infty,\N^*}&\leq \sup_{x\in \N^*}\left(1+\frac{u(x-1)}{u(x)}\right)\|f/u\|_\infty.
\label{eq:weighted_trivial_majo}
\end{align}
Under $\mathbf{H_1}$, as $\|\p_u \varphi_i\|_\infty \leq 1$, it implies that  
$$\sup_{x\in \N^*}|\p_u\p^* \varphi_i(x)|\leq \sup_{x \in \N^*}\left(1+\frac{u(x-1)}{u(x)}\right).$$
Plugging this in the equation (\ref{equation_majoration_stein_factor}) yields the first upper bound of the theorem. 

On the other hand, if $u(x)=q^x$ on $\N$ with $q\geq 1$, then by using that $\1_{[0,i]}=1-\1_i-\1_{[i+1,\infty)}$, we write 
\begin{eqnarray*}
\p_u^*\p \varphi_i(j)&= &\left(1-\frac{u(j-1)}{u(j)}\right)+2\frac{u(j-1)}{u(j)}\1_{j=i}+2\left(\frac{u(j-1)}{u(j)}-1\right)\1_{j\geq i+1}\\
&\leq & 1-\frac{1}{q}+2\frac{1}{q}\1_{j=i}
\end{eqnarray*}
which proves the second upper bound.
\end{proof}

%  \ % \\

%\textit{Half-line indicator functions and the Kolmogorov distance.} 
\subsection{Approximation in Kolmogorov distance}

\begin{proof}[Proof of Theorem \ref{thm:stein_factor_kolmo_order1}]
As one can see in Lemma \ref{lemma:argmaxtv}, the function that realizes the maximum in the first Stein factor associated to bounded functions, $f=\1_{[0,i]}$, is also an element of the class of the half-line indicator functions. Hence without further analysis the analogous of Lemma \ref{thm:stein_factor_tv_0} and Theorem \ref{thm:stein_factor_tv_order1} hold by replacing $\mathcal{F}=\left\{0\leq f\leq 1\right\}$ by $\mathcal{F}=\left\{\1_{[0,m]},\,m\in \N\right\}$.
\end{proof}

For the second Stein factor, we begin by determining the argmax of the pointwise factor, as we did previously.

\begin{lemma}[Argmax of the pointwise Stein factor]
\label{lemma_argmax_kolmo}
For all $i\in \N$
$$\max\left\{-\p g_{\1_{[0,i-1]}}(i),  \p g_{\1_{[0,i]}}(i)  \right\}=\sup_{f=\1_{[0,m]},\,m\in \N}|\p g_f(i)|.$$
\end{lemma}

\begin{proof}
Let $f=\1_{[0,m]}$ for an integer $m$. By Lemma \ref{lemma:bx1}, if $m\leq i-1$,
$$\p g_g(i)=\sum_{j=0}^m \pi(j)(e_{i+1}^- -e_i^-).$$
Hence by Lemma \ref{lemma:bx2},
$$|\p g_f(i)|=-\p g_g(i)=(e_{i}^- -e_{i+1}^-)\sum_{j=0}^m \pi(j),$$
so the maximum when $m$ browses the interval $[0,i-1]$ is attained in $m=i-1$. 

Now, if $m \geq i$, let us call $F=1-f=\1_{[m+1,\infty)}$. By the same lemmas,
\begin{align*}
|g_f(i)|&=|g_F(i)|=|e_i^+-e_{i-1}^+|\sum_{j=m+1}^\infty{\pi(j)}=(e_i^+-e_{i-1}^+)\sum_{j=m+1}^\infty{\pi(j)}=g_f(i),
\end{align*}
so the maximum when $m$ browses the interval $[i,+\infty)$ is attained in $m=i$.
\end{proof}

\begin{lemma}[Second pointwise Stein's factor for indicator functions]
\label{thm:majo_stein_order1_kolmo}
\begin{itemize}
\item[]

\item Under $\mathbf{H_1}$, for all integer $i\in \N^*$, the quantity $\sup_{f=\1_{[0,m]},\,m\in \N} \p g_f(i)$ is bounded by the maximum of
\begin{eqnarray*}
\int_0^\infty{e^{-\sigma(1,*u)t}\left(\P(X_{1,*u,t}^i=i)-\frac{u(i)}{u(i-1)}\P(X_{1,*u,t}^i=i-1)\right)dt}
\end{eqnarray*}
and 
\begin{align*}
\int_0^\infty{e^{-\sigma(1,*u)t}\left(\P(X_{1,*u,t}^i=i)-\frac{u(i)}{u(i+1)}\P(X_{1,*u,t}^i=i+1)\right)dt}.
\end{align*}

\item Under $\mathbf{H_2}$, for all integer $i\in \N$, the quantity $\sup_{f=\1_{[0,m]},\,m\in \N} \p g_f(i+1)$ is bounded by the maximum of
\begin{align*}
\int_0^\infty{e^{-\sigma(1,u)t}\left(\P(X_{1,u,t}^i=i)-\frac{u(i)}{u(i-1)}\P(X_{1,u,t}^i=i-1)\right)dt}
\end{align*}
and 
\begin{align*}
\int_0^\infty{e^{-\sigma(1,u)t}\left(\P(X_{1,u,t}^i=i)-\frac{u(i)}{u(i+1)}\P(X_{1,u,t}^i=i+1)\right)dt}.
\end{align*}
\end{itemize}
Moreover, if the potential $V_{1,*u}$ (respectively $V_{1,u}$) is constant, then the first (respectively the second) upper bound is in fact an equality.
\end{lemma}

\begin{proof}%[Proof of Theorem \ref{thm:majo_stein_order1_kolmo}]
If $f=\1_{[0,m]}$ then $\p_u^*\p f_m=\frac{1}{u(m)}\1_m-\frac{1}{u(m+1)}\1_{m+1}$. Under $\mathbf{H_1}$, by equation \eqref{eq:sol_stein_1} and Theorem \ref{theorem_intertwining_order2_hard},
\begin{align*}
-\p g_{\1_{[0,i-1]}}(i)&=u(i)\int_0^\infty{P_{1,*u,t}^{V_{1,*u}}(-\frac{1}{u(i-1)}\1_{i-1}+\frac{1}{u(i)}\1_{i})dt}\\
&\leq \int_0^\infty{e^{-\sigma(1,*u)t}\left(\P(X_{1,*u,t}^i=i)-\frac{u(i)}{u(i-1)}\P(X_{1,*u,t}^i=i-1)\right)dt}.
\end{align*}
Similarly, 
\begin{align*}
\p g_{\1_{[0,i]}}(i)&\leq \int_0^\infty{e^{-\sigma(1,*u)t}\left(\P(X_{1,*u,t}^i=i)-\frac{u(i)}{u(i+1)}\P(X_{1,*u,t}^i=i+1)\right)dt}.
\end{align*}
We get the conclusion by Lemma \ref{lemma_argmax_kolmo}. The proof is analogous under $\mathbf{H_2}$, using this time equation \eqref{eq:sol_stein_1bis} and Theorem \ref{theorem_intertwining_order2_easy}.
\end{proof}

Finally, Theorem \ref{thm:stein_factor_kolmo_order2} and \ref{thm:stein_factor_kolmo_order2_bis} are simple consequences of the previous lemma.

\section{Proof of Section \ref{sect:exe}}

\label{sect:proof_expl}

\label{sect:proof_expl_mminfty}
The second upper bound of Lemma \ref{lem:majo_ptw_mminfty} derives by classical arguments from Mehler's formula \eqref{eq:mehler-poisson} and the following lemma.
\begin{lemma}[Upper bound on differences of the pointwise probabilities of the Poisson distribution]
\label{lem:majo_ptw_poisson}
\begin{align}
%\sup_{x \in \N}\,\mathcal{P}_\lambda(x) & \leq 1 \wedge \frac{c}{\sqrt{\l}}, &c&:= \frac{1}{\sqrt{2e}}, \label{eq:majo_poisson_1}\\
\sup_{x \in \N}\,\left|\mathcal{P}_\lambda(x)-\mathcal{P}_\lambda(x-1)\right| & \leq 1 \wedge \frac{C}{\l}, &C&:=\frac{1}{\sqrt{2\pi}}e^{\frac{1}{\sqrt{2}}} \leq 1.\label{eq:majo_poisson_2}
\end{align}
\end{lemma}

\begin{proof}[Proof of Lemma \ref{lem:majo_ptw_poisson}]
Set 
\begin{align*}
q(\l,x)&= \lambda|\mathcal{P}_\lambda(x)-\mathcal{P}_\lambda(x-1)|, &x\in\N,\quad\l>0.
\end{align*}
Let us show that
$$\sup_{ x\in \N,\, \lambda>0}{q(\lambda,x)}<+\infty.$$
Firstly, 
\begin{align*}
q(\l,x)&= \frac{\l^xe^{-\l}}{x!}|\l-x|, &x\in\N,\quad\l>0.
\end{align*}
We first deal with the case where $x\in \N^*$. By a formula of Robbins (\cite{robbins1955stirling}), we know that for all $x \in \N^*$,
\begin{align*}
x!>\sqrt{2\pi x}\, x^x \, e^{-x+\frac{1}{12x}} \geq \sqrt{2\pi}\,e^{\frac{1}{2}\log x+x\log x-x}.
\end{align*}
Hence, $q(\l,x) \leq \frac{1}{\sqrt{2\pi}}e^{f(\l,x)}$ with 
\begin{align*}
f(\l,x)&=x-\l+\log|x-\l|-\frac{1}{2}\log x +x \log \frac{\l}{x},\\
\p_{\l} f(\l,x)&=-1+\frac{x}{\l}+\frac{1}{\l-x}.
\end{align*}
In the sequel we derive upper bounds of $f(\l,x)$ on relevant subsets of $(0,\infty)\times [1,\infty)$. One finds that 
\begin{align*}
\p_{\l}f(\l,x)=0&\Leftrightarrow (\l-x)(x-\l+\sqrt{\l})(x-\l-\sqrt{\l})=0.
\end{align*}
Let us call $\l_1(x)$ the solution of the equation $x=\l+\sqrt{\l}$ and $\l_2(x)$ the solution of the equation $x=\l-\sqrt{\l}$. We have $0<\l_1(x)<x<\l_2(x)$. 

If $\l\leq x$, then at $x$ fixed the function $f(\l,x)$ is increasing on $(0,\l_1(x)]$ and decreasing on $[\l_1(x),x]$. Hence,
\begin{align*}
\sup_{x \geq 1,\, 0<\l \leq x} f(\l,x) &=\sup_{x \geq 1} f(\l_1(x),x)=\sup_{\l >0}f(\l, \l+\sqrt{\l}).
\end{align*}
Moreover, using that $\forall u \geq 0, \,\log (1+u) \geq u -u^2/2$, we find that
\begin{align}
f(\l, \l+\sqrt{\l})&=\sqrt{\l}+\frac{1}{2}\log\frac{\l}{\l+\sqrt{\l}}+(\l+\sqrt{\l})\log\frac{\l}{\l+\sqrt{\l}}\nonumber\\
&=\sqrt{\l}-(\l+\sqrt{\l}+\frac{1}{2})\log \left(1+\frac{1}{\sqrt{\l}}\right) \label{eq:majo_stirling}\\
&\leq \sqrt{\l}-(\l+\sqrt{\l}+\frac{1}{2})\left(\frac{1}{\sqrt{\l}}-\frac{1}{2\l}\right) \nonumber\\
&=-\frac{1}{2}\left(1-\frac{1}{2\l}\right).\nonumber
\end{align}
Hence, if $\l \geq \frac{1}{2}$ then $f(\l, \l+\sqrt{\l})\leq 0$. If $\l \leq \frac{1}{2}$, by going back up to the equation \eqref{eq:majo_stirling}, $$f(\l, \l+\sqrt{\l})\leq \sqrt{\l}\leq \frac{1}{\sqrt{2}}.$$ At the end,
\begin{align*}
\sup_{x \in \N^*,\, 0<\l \leq x} f(\l,x) &\leq \frac{1}{\sqrt{2}}.
\end{align*}
Let us call $C_1:=\frac{1}{\sqrt{2\pi}}e^{\frac{1}{\sqrt{2}}} \sim 0,8$.

Now let us deal with the case where $\l \geq x$. We apply a different strategy for small integers $x$ as for large integers $x$. First of all, for all $x\in \N^*$ and for all $\l \geq x$,
\begin{align*}
q(\l,x) &= \frac{1}{x!}e^{-\l}\l^x(\l-x)\leq \frac{1}{x!}e^{-\l}\l^{x+1}
\end{align*}
and it is easy to see that at $x$ fixed the maximum of the right-hand expression is attained at $\l=x+1$. Hence
\begin{align*}
q(\l,x) &\leq \frac{1}{x!}e^{-(x+1)}(x+1)^{x+1}.
\end{align*}
Hence
\begin{align*}
\sup_{x \in \left\{1,2,3\right\},\, \l \geq x} q(\l,x)&\leq C_2:= \max_{x \in \left\{1,2,3\right\}}\;{ \frac{1}{x!}e^{-(x+1)}(x+1)^{x+1}} \sim 0.7.
\end{align*}
On the other hand, by the same reasoning as below, we find that
\begin{align*}
\sup_{x \in [4,+\infty),\, \l \geq x} f(\l,x) &=\sup_{x \in [4,+\infty)} f(\l_2(x),x)=\sup_{\l \geq 4}\, f(\l, \l-\sqrt{\l}).
\end{align*}
Now, for all $\lambda >1$,
\begin{align*}
f(\l,\l-\sqrt{\l})&=-\sqrt{\l}-\left(\frac{1}{2}+\l-\sqrt{\l}\right)\log\left(1-\frac{1}{\sqrt{\l}}\right).
\end{align*}
We use that $\forall u \in [0, \frac{1}{2}], -\log(1-u)\leq u +u^2$. As $\l \geq 4$ implies $\frac{1}{\sqrt{\l}}\leq \frac{1}{2}$, 	
\begin{align*}
f(\l, \l-\sqrt{\l})&\leq -\sqrt{\l}+\left(\frac{1}{2}+\l-\sqrt{\l}\right)\left(\frac{1}{\sqrt{\l}}+\frac{1}{\l}\right)\\
&=-\frac{1}{2\sqrt{\l}}\left(1-\frac{1}{\sqrt{\l}}\right)\\
&\leq 0.
\end{align*}

At the end,
\begin{align*}
\sup_{\l \geq x \geq 4} f(\l,x) &\leq C_3:=\frac{1}{\sqrt{2\pi}} \sim 0.4.
\end{align*}

It remains the case where $x=0$, for which it is trivial to see that
\begin{align*}
q(\l,0)&=\l e^{-\l}\leq C_4=e^{-1},&\l >0.
\end{align*}

The final result follows with $C=\max\left\{C_1,C_2,C_3,C_4\right\}=C_1$.
\end{proof}

\bibliographystyle{plainnat}
\bibliography{CloezDelplancke2017Biblio}
\end{document}